\documentclass[12pt]{amsart}

\usepackage[english]{babel}
\usepackage[latin1]{inputenc}
\usepackage[T1]{fontenc}
\usepackage{lmodern}
\usepackage{amssymb}
\usepackage{amsmath}
\usepackage{enumerate}

\newtheorem{theo}{Theorem}[section]
\newtheorem{lemm}[theo]{Lemma}
\newtheorem{prop}[theo]{Proposition}
\newtheorem{coro}[theo]{Corollary}
\newtheorem{defi}[theo]{Definition\rm}
\newtheorem{rema}[theo]{Remark}

\newenvironment{theoi}[2][Theorem]{\begin{trivlist}
\item[\hskip \labelsep {\bfseries #1}\hskip \labelsep {\bfseries #2}]}{\end{trivlist}}

\newenvironment{propi}[2][Proposition]{\begin{trivlist}
\item[\hskip \labelsep {\bfseries #1}\hskip \labelsep {\bfseries #2}]}{\end{trivlist}}

\def\og{\leavevmode\raise.3ex\hbox{$\scriptscriptstyle\langle\!\langle$~}}
\def\fg{\leavevmode\raise.3ex\hbox{~$\!\scriptscriptstyle\,\rangle\!\rangle$}}

\newcommand{\Z}{\mathbb{Z}}

\newcommand{\R}{\mathbb{R}}
\newcommand{\C}{\mathbb{C}}

\renewcommand{\Re}{\mathfrak{Re}}
\renewcommand{\Im}{\mathfrak{Im}}

\def\germ #1 {\mathfrak{#1}}
\def\cal #1 {\mathcal{#1}}

\title[Unitary representations of $\tilde{SU}(1,1)$ and tensor products]{Unitary representations of the universal cover of $SU(1,1)$ and tensor products}

\author{Guillaume Tomasini,\ Bent \O rsted}

\begin{document}


\begin{abstract}
In this paper we initiate a study of the relation between weight modules for simple Lie algebras and unitary representations of the corresponding simply-connected Lie groups. In particular we consider in detail from this point of view the universal covering group of $SU(1,1)$, including new results on the discrete part of tensor products of irreducible representations. As a consequence of these results, we show that the set of smooth vectors of the tensor product intersects trivially some of the representations in the discrete spectrum.
\end{abstract}


\maketitle


\section{Introduction}


The category of weight modules for simple Lie algebras, and in particular those of degree one, has been much studied in recent years. From the point of view of the unitary dual of the corresponding simply-connected Lie group, it is a natural question to find those degree one modules that integrate to unitary representations; they should form a small but interesting class of unitary representations with small Gelfand-Kirillov dimension. In this paper we treat in detail the case of $\germ sl (2,\C)$, in effect giving a new proof of the classification due to Pukansky of the unitary dual of the universal covering of $SU(1,1)$; furthermore we apply this to studying in detail tensor products of such representations, obtaing new results about the discrete spectrum in such tensor products, and about the possible relation between the smooth vectors of the tensor product and the representations in the discrete spectrum. The methods are developed so as to apply to higher rank cases, where similar results are expected to hold.

Let us now review the main results of this paper. Let $G$ denote the uiversal covering of $G_0=SU(1,1)$. Let $H=\left(\begin{array}{cc}
 1 & 0\\
 0 & -1
\end{array}\right)$. Then $\{e^{itH},\ 0\leq t<2\pi\}$ generates a maximal compact subgroup $K_0$ of $G_0$. Its covering group is $K=\{\exp (itH),\ t\in\ \R\}$. The center of $G$ is generated by $\exp (2i\pi H)$. Let now $\rho$ denote an irreducible unitary representation of $G$. From Schur's lemma we conclude that $\rho(\exp (2i\pi H))=e^{-2i\pi\tau_0}I$. Therefore, $\tilde{\rho}(\exp(itH)):=e^{i\tau_0t}\rho(\exp(itH))$ is a unitary representation of $\R$, with period $2\pi$, and hence is completely reducible. As a consequence, $H$ possesses a complete system of eigenelements. In other word, the corresponding representation of the complexified Lie algebra $\germ sl (2,\C)$ is a weight module (see definition \ref{def:WM}). We shall review the basics of weight module in section 2.

In section 3, we classify the unitarisable weight modules for $\germ su (1,1)$. Recall that a unitarisable module is a module defined on a Hilbert space which is the differential of a unitary module for the universal covering group $G$. Using the explicit action of $\germ sl (2,\C)$, we recover the classification due to Pukansky of the unitary dual of $G$, which falls into 3 series: the principal series $\pi_{\epsilon,it}$ ($0< \epsilon\leq 1$, $t\in \R$), the complementary series $\pi^c_{\sigma,\tau}$ ($0< \sigma,\ \tau< 1$), the (continuation of the) discrete series $\pi^{\pm}_{\lambda}$ ($\lambda>0$), and the extra trivial representation.

In section 4 and 5, we study a tensor product $V$ of the form: $\pi_1\otimes\pi^+_{\lambda}$ where $\pi_1$ is $\pi_{\epsilon,it}$, $\pi^c_{\sigma,\tau}$ or $\pi^-_{\mu}$. The main result in section 4 is theorem \ref{thm:tens-alg}:
\begin{theoi}{\ref{thm:tens-alg}$'$}
{\it Every simple weight $\germ sl (2,\C)$-module $W$ whose support is included in the support of $V$ appears as a quotient of the algebraic tensor product $V$.}
\end{theoi}

Then section 5 is devoted to the study of the discrete part in the Hilbert space $V$ (the completion of the algebraic tensor product $V$). The main result is theorem \ref{thm:disc-spec}. Let us state a particular case of this theorem.
\begin{theoi}{\ref{thm:disc-spec}$'$}
\begin{enumerate}
 \item {\it If $0<\mu+\lambda<1$, then the Hilbert representation $\pi^-_{\mu}\otimes \pi^+_{\lambda}$ contains the representation $\pi^c_{\lambda,\mu}$,  belonging to the complementary series.}
 \item {\it If $0<\sigma+\tau+\lambda<1$, then the Hilbert representation $\pi^c_{\sigma,\tau}\otimes \pi^+_{\lambda}$ contains the representation $\pi^c_{\sigma+\lambda,\tau}$, belonging to the complementary series.}
 \item {\it If $1<\sigma+\tau-\lambda<2$, then the Hilbert representation $\pi^c_{\sigma,\tau}\otimes \pi^+_{\lambda}$ contains the representation $\pi^c_{\sigma,\tau-\lambda}$, belonging to the complementary series.}
 \end{enumerate}
\end{theoi}
In section 5, we also give an explicit generator for all submodules in the discrete spectrum of $V$. As a consequence we prove proposition \ref{prop:Cinfty}. A particular case of this proposition in the above setting is the following 
\begin{propi}{\ref{prop:Cinfty}$'$}
{\it If $0<\lambda+\mu<1$ (resp. $0<\sigma+\tau+\lambda<1$, and $1<\sigma+\tau-\lambda<2$), then the Hilbert submodule $\pi^c_{\lambda,\mu}$ (resp. $\pi^c_{\sigma+\lambda,\tau}$, and $\pi^c_{\sigma,\tau-\lambda}$) of the Hilbert representation $\pi^-_{\mu}\otimes \pi^+_{\lambda}$ (resp. $\pi^c_{\sigma,\tau}\otimes \pi^+_{\lambda}$) intersects trivially the set of smooth vectors in $\pi^-_{\mu}\otimes \pi^+_{\lambda}$ (resp. $\pi^c_{\sigma,\tau}\otimes \pi^+_{\lambda}$).}
\end{propi}

The proofs involve the algebraic structure of weight modules and asymptotic analysis of hypergeometric functions.


\section{Weight modules}\label{sec:}


Let $\germ g $ denote a reductive Lie algebra and $\cal U (\germ g )$ denote its universal enveloping algebra. Let $\germ h $ be a fixed Cartan subalgebra and denote by $\cal R $ the corresponding set of roots. For $\alpha \in \cal R $, we denote by $\germ g _{\alpha}$ the root space for the root $\alpha$.


\subsection{The category of weight modules}\label{ssec:WM}


\begin{defi}\label{def:WM}
A $\germ g $-module $M$ is a \emph{weight module} if it is finitely generated, and $\germ h $--diagonalizable in the sense that 
$$M=\oplus_{\lambda \in \germ h ^*}\: M_{\lambda}, \quad \mbox{where } M_{\lambda}=\{m\in M \: : \: H\cdot m = \lambda(H)m, \: \forall \: H\in \germ h \},$$ with weight spaces $M_{\lambda}$ of finite dimension.
\end{defi}
\begin{rema}
Note that we require finite dimensional weight spaces in our definition, which is not always the case in the literature. This category also appears as a particular case of several other categories (e.g. \cite{PS02,PZ04a} or \cite{DFO94,FMO10}).
\end{rema}

 The set of all weight modules forms a full subcategory of the category of all modules, denoted by $\cal M  (\germ g , \germ h )$. 
Given a weight module $M$, we call \emph{support} of $M$ the set
$$Supp(M)=\{\lambda \in\ \germ h ^*\ :\ M_{\lambda}\not=0\}.$$
The \emph{degree} of a weight module $M$ is the (possibly infinite) number 
$$deg(M)=\sup_{\lambda \in \germ h ^*} \: \{dim(M_{\lambda})\}.$$ For instance, a \emph{degree one module} is a weight module whose all non zero weight spaces are 1-dimensional. Such modules have been classified by Benkart, Britten and Lemire in \cite{BBL97}. They will be the main object of investigation of this paper.


\subsection{The modules of degree $1$}\label{ssec:deg1}


Let us review the classification of degree one modules for simple Lie algebras. First we have the following

\begin{theo}[Benkart, Britten, Lemire {\cite[prop. 1.4]{BBL97}}]
 Let $\germ g $ be a simple Lie algebra. Let $M$ be a simple infinite dimensional degree 1 weight module. Then
 \begin{enumerate}
 \item The Lie algebra $\germ g $ is of type $A$ or $C$.
 \item The Gelfand-Kirilov dimension of $M$ is given by the rank of $\germ g $.
\end{enumerate}
\end{theo}

\subsubsection{Modules over the Weyl algebra}\label{sssec:Weyl}

Let $n$ be a positive integer. Recall that the Weyl algebra $W_{n}$ is the associative algebra generated by the $2n$ generators $\{q_{i}, \: p_{i}, \: 1\leq i\leq n\}$ submitted to the following relations:
$$[q_{i},q_{j}]=0=[p_{i},p_{j}],\quad [p_{i},q_{j}]=\delta_{i,j}\cdot 1,$$ where the bracket is the usual commutator for associative algebras.

Define a vector space as follows. Fix some $a\in \C^n$. Let $$\cal K (a)=\left\{k\in \Z^n \: :\: \mbox{if } a_{i}\in \Z, \mbox{ then } a_{i}+k_{i}<0 \iff a_{i}<0\right\}.$$ Now our vector space $W(a)$ is the $\C$-vector space whose basis is indexed by $\cal K (a)$. For each $k\in \cal K (a)$, we fix a basis vector $x(k)$. Let $(\epsilon_i)_{1\leq i\leq n}$ denote the canonical basis of $\Z^n$. Define an action of $W_{n}$ on $W(a)$ by the following recipe:
$$\begin{array}{ccl}
q_{i}\cdot x(k) & = & \left\{ \begin{array}{lc} (a_{i}+k_{i}+1)x(k+\epsilon_{i}) & \mbox{if } a_{i}\in \Z_{<0}\\ x(k+\epsilon_{i}) & \mbox{otherwise} \end{array}\right. ,\\
 & & \\
p_{i}\cdot x(k) & = & \left\{ \begin{array}{lc} x(k-\epsilon_{i}) & \mbox{if } a_{i}\in \Z_{<0}\\ (a_{i}+k_{i})x(k-\epsilon_{i}) & \mbox{otherwise} \end{array}\right.
\end{array}$$
This basis shall be refered to as the \emph{standard basis} of $W(a)$.

Then we have:
\begin{theo}[Benkart,Britten,Lemire {\cite[thm 2.9]{BBL97}}]\label{thmBBLWeyl}
Let $a\in \C^n$. Then $W(a)$ is a simple $W_{n}$-module.
\end{theo}

\subsubsection{Type A case}\label{sssec:deg1A}

In this section only, $\germ g $ denotes a simple Lie algebra of type $A$. We shall construct weight $\germ g $-modules of degree $1$ by using the previous construction. We realize the Lie algebra $\germ g $ inside some $W_{n}$. Let $n-1$ be the rank of $\germ g $. Then, we can embed $\germ g $ into $W_{n}$ as follows: to an elementary matrix $E_{i,j}$ we associate the element $q_{i}p_{j}$ of $W_{n}$. This is easily seen to define an embedding of $\germ g $ into $W_{n}$.  Let $\cal K _{0}(a)=\left\{k\in \cal K (a)\: :\: \displaystyle\sum_{i=1}^n\: k_{i}=0\right\}$. Let $N(a)$ be the subspace of $W(a)$ whose basis is indexed by $\cal K _{0}(a)$. Then we have the following:

\begin{theo}[Benkart,Britten,Lemire {\cite[thm 5.8]{BBL97}}]\label{thm:BBLA}\hfill{ }
\begin{enumerate}
\item The vector subspace $N(a)$ of $W(a)$ is a simple weight $\germ sl (n,\C) $-module of degree $1$.
\item Conversely if $M$ is a simple weight $\germ sl (n,\C) $-module of degree $1$, then there exist $a=(a_1,\ldots, a_n)\in \C^n$, such that the module $M$ is isomorphic to $N(a)$.
\end{enumerate}
\end{theo}

\subsubsection{Type C case}\label{sssec:deg1C}

In this section only, $\germ g $ denotes a simple Lie algebra of type $C$. We shall construct weight $\germ g $-modules of degree $1$ in the same way as above. So we need to realize the Lie algebra $\germ g $ inside some $W_{n}$. Let $n$ be the rank of $\germ g $. Then, $span_{\C}\{q_{i}p_{j}, p_{i}p_{j}, q_{i}q_{j}, \: 1\leq i,j\leq n\}$ is a subalgebra of $W_{n}$ isomorphic to $\germ g $. More specifically, the Cartan subalgebra is given by 
$$span_{{\C}}\left(\left\{q_{i}p_{i}-q_{i+1}p_{i+1}, \: i=1,\ldots,n-1\right\}\cup\left\{q_{n}p_{n}+\frac{1}{2}\right\}\right),$$ the $n-1$ weight vectors corresponding to the short simple roots are given by $q_{i}p_{i+1}$ with $i=1,\ldots ,n-1$, and the weight vector corresponding to the long simple root is given by $\frac{1}{2}q_{n}^2$. Note that this is not the same kind of embedding as for Lie algebras of type A. 

Let $\cal K _{\bar 0}(a)=\left\{k\in \cal K (a)\: :\: \displaystyle\sum_{i=1}^n\: k_{i}\in 2\Z\right\}$. Let $M(a)$ be the subspace of $W(a)$ whose basis is indexed by $\cal K _{\bar 0}(a)$.Then we have the following:

\begin{theo}[Benkart,Britten,Lemire {\cite[thm 5.21]{BBL97}}]\label{thm:BBLC}\hfill{ }
\begin{enumerate}
\item The vector subspace $M(a)$ of $W(a)$ is a simple weight $\germ sp (n,\C) $--module of degree $1$.
\item Conversely if $M$ is an infinite dimensional simple weight $\germ sp (n,\C) $--module of degree $1$, then there exists $a=(a_{1},\ldots , a_{n})\in \C^n$ such that $M\cong M(a)$. 
\end{enumerate}
\end{theo}


\subsection{The case of $\mathfrak{sl}(2,\C)$}\label{ssec:sl2}


In this section, we review the classification of all weight modules for $\germ g =\germ sl (2,\C)$. We shall consider the standard $\germ sl (2,\C)$-triple $(F,H,E)$, given by:
$$F=\left(\begin{array}{cc}
 0 & 0\\
 1 & 0
\end{array}\right),\ H=\left(\begin{array}{cc}
 1 & 0\\
 0 & -1
\end{array}\right),\ E=\left(\begin{array}{cc}
 0 & 1\\
 0 & 0
\end{array}\right).$$
 We therefore have the following relations:
$$[H,E]=2E, \quad [H,F]=-2F, \quad [E,F]=H.$$

\begin{prop}\label{prop:WMsl2}
Let $M$ be a simple weight $\germ sl (2,\C)$-module. Then $deg(M)=1$.
\end{prop}
\begin{proof}
Recall that $\Omega = \frac{1}{4}H^2+\frac{1}{2}H+FE$ is in the center of the universal enveloping algebra of $\germ sl (2,\C)$. Therefore, $M$ being simple, $\Omega$ acts as a scalar operator. On the other hand, as $M$ is a weight module, $H$ acts on each weight space by some constant (the weight). Therefore, on each weight space, $FE$ acts by some constant. From this, we conclude that $\cal U (\germ g )_{0}$, the commutant of $\C H$, acts by some constant on each weight space. But, since $M$ is simple, given two non zero vectors $v$ and $w$ in the same weight space, there should exist some element $u\in \cal U (\germ g )$ sending $v$ to $w$. The fact that $v$ and $w$ have the same weight forces $u$ to be in the commutant of $\C H$. From the above we know that $u$ acts by some constant. This forces $v$ and $w$ to be proportional and therefore the corresponding weight space is $1$-dimensional. This completes the proof.

\end{proof}

For a simple weight module $M$, the action of $\Omega$ on $M$ is called the \emph{infinitesimal character}. From theorem \ref{thm:BBLA}, the simple weight modules are indexed by $a=(a_{1},a_{2})\in \C^2$. Recall that we set $$\cal K _0(a)=\left\{k\in\Z^2\ :\ \mbox{if } a_i\in\Z, \mbox{ then } a_i+k_i<0 \iff a_i<0\right\}.$$ This reduces here to
$$\cal K _0(a)=\left\{k\in \Z \: :\: \mbox{if } a_{i}\in \Z, \mbox{ then } a_{i}+(-1)^{i-1}k<0 \iff a_{i}<0\right\}.$$ Recall then that the weight module $N(a)$ has a basis $x(k)$ indexed by $\cal K _0(a)$. We shall consider 4 cases:
\begin{enumerate}[(I)]
 \item Both $a_1$ and $a_2$ are not negative integers.
 \item $a_1$ is not a negative integer but $a_2$ is a negative integer.
 \item $a_2$ is not a negative integer but $a_1$ is a negative integer.
 \item Both $a_1$ and $a_2$ are negative integers.
\end{enumerate}
Then we have the following action of $\germ g $ on $N(a)$:
$$\begin{array}{cl}
(I) & \left\{\begin{array}{ccl}
H\cdot x(k) & = & (a_{1}-a_{2}+2k)x(k),\\
E\cdot x(k) & = & (a_{2}-k)x(k+1),\\
F\cdot x(k) & = & (a_{1}+k)x(k-1).
\end{array}\right.\\
 & \\
(II) & \left\{\begin{array}{ccl}
H\cdot x(k) & = & (a_{1}-a_{2}+2k)x(k),\\
E\cdot x(k) & = & x(k+1),\\
F\cdot x(k) & = & (a_{1}+k)(a_2-k+1)x(k-1).
\end{array}\right.\\
 & \\
(III) & \left\{\begin{array}{ccl}
H\cdot x(k) & = & (a_{1}-a_{2}+2k)x(k),\\
E\cdot x(k) & = & (a_1+k+1)(a_{2}-k)x(k+1),\\
F\cdot x(k) & = & x(k-1).
\end{array}\right.\\
\end{array}$$
$$\begin{array}{cl}
(IV) & \left\{\begin{array}{ccl}
H\cdot x(k) & = & (a_{1}-a_{2}+2k)x(k),\\
E\cdot x(k) & = & (a_1+k+1)x(k+1),\\
F\cdot x(k) & = & (a_{2}-k+1)x(k-1).
\end{array}\right.
\end{array}$$


\section{Unitarisability}


We keep the previous notations. Let $\germ g =\germ sl (2,\C)=span_{\C}\{H,E,F\}$. Set $h=-i\left(E-F\right)$, $e=\frac{1}{2}\left(-iH+E+F\right)$ and $f=\frac{1}{2}\left(iH+E+F\right)$. Then $span_{\R}\{h,e,f\}$ is a real Lie algebra isomorphic to $\germ su (1,1)$.

Let $G$ denote the simply connected Lie group with Lie algebra $\germ su (1,1)$. Recall the following result of Nelson \cite[cor. 9.3]{Ne59}:

\begin{theo}[Nelson]\label{thm:nelson}
 Let $\rho$ be a representation of $\germ su (1,1)$ on a Hilbert space by skew-symmetric operators with domain $\germ D $. Then there is a unitary representation $U$ of $G$ such that $\germ D $ is the space of infinitely differentiable vectors for $U$ and $dU(X)=\rho(X)$, $\forall\ X \in\ \germ g $ if and only if 
 $$A=\rho(h)^2+\rho(e)^2+\rho(f)^2$$ is essentially self-adjoint and $\germ D = \displaystyle\bigcap_{k=1}^{\infty}\ \germ D ({\bar A}^k)$, $\bar{A}$ being the closure of $A$.
\end{theo}

Remark that we have $A=\rho(\Omega)-\frac{1}{2}\left(\rho(E)-\rho(F)\right)^2.$ A $\germ g $-module giving rise to a representation $\rho$ of $\germ su (1,1)$ satisfying the assumptions of Nelson's theorem will be refer to as a \emph{unitarisable} module.

Thanks to this theorem, to find which $N(a)$ are unitarisable we need to construct on $N(a)$ a Hilbert space structure such that $h$, $e$, and $f$ acts by skew-symmetric operators. It is then equivalent to construct on $N(a)$ a Hilbert space structure such that $H^*=H$, $E^*=-F$, and $F^*=-E$. Let $\langle\cdot ,\cdot\rangle$ be an inner product on $N(a)$. By construction, $H$ acts on $N(a)$ by a semisimple operator. So, for $H$ to be self-adjoint it is necessary that weight vectors for different weight are orthogonal and that all the weights are real numbers. This means that the basis $\{x(k)\}_{k\in\ \cal K _0}$ is an orthogonal basis and that $a_1-a_2 \in\ \R$.

Besides we must also have 
$$\langle F\cdot x(k+1),x(k)\rangle = - \langle x(k+1),E\cdot x(k)\rangle, \ \forall\ k \in\ \cal K _0.$$
We shall work with this condition in the different cases $(I)$, $(II)$, $(III)$, and $(IV)$.

\subsection{Case $(I)$}

In this case, the condition becomes:
 $$(a_1+k+1)\|x(k)\|^2=-(\bar a_2 -k)\|x(k+1)\|^2.$$
Let us consider several possibilities:
\begin{enumerate}[(i)]
 \item Assume that both $a_1$ and $a_2$ are not integers. In this case, $\cal K _0=\Z$ and we have $a_1+k+1\not=0$ and $\bar a_2 -k\not=0$. So, for the condition to hold it is necessary and sufficient that 
 $$\forall\ k \in\ \Z,\ \frac{k-\bar a_2 }{k+a_1+1} \in\ \R_{>0}.$$ But we have seen that $a_1-a_2 \in\ \R$, so we can set $a_1=a_2+r$ for some $r \in\ \R$. Therefore we must have either $\Im(a_2)=0$ or $2\Re(a_2)+r+1=0$. In the first situation we must also have 
 $$\forall\ k \in\ \Z,\ \frac{k-a_2 }{k+a_2+r+1} > 0.$$ This is true if and only if 
 $$-2-[a_2]<a_2+r<-1-[a_2],$$ where $[a_2]$ is the integer such that $[a_2]\leq a_2<[a_2]+1$. Then we can express $\|x(k)\|^2$ uniquely in terms of $\|x(0)\|^2$, via the formula:
 \begin{subequations}\label{norm:cpl}
 
 \begin{align}
  \|x(k)\|^2 = & \frac{\displaystyle\prod_{j=1}^k\ (j+a_1)}{\displaystyle\prod_{j=1}^k\ (j-1-a_2)}\|x(0)\|^2,\ \mbox{if } k>0,\\
  \|x(k)\|^2 = & \frac{\displaystyle\prod_{j=1}^{-k}\ (j+a_2)}{\displaystyle\prod_{j=1}^{-k}\ (j-1-a_1)}\|x(0)\|^2,\ \mbox{if } k<0
 \end{align}

 \end{subequations}

 Conversely, if we define an inner product on $N(a)$ such that $\{x(k)\}$ is an orthogonal basis satisfying formulae \eqref{norm:cpl}, then Nelson theorem applies and thus the corresponding module is unitarisable.
 
 In the second situation, we have 
  $$\forall\ k \in\ \Z,\ \frac{k-\bar a_2 }{k+a_1+1}=1 \in\ \R_{>0}.$$ Then we can express $\|x(k)\|^2$ uniquely in terms of $\|x(0)\|^2$, via the formula:
  \begin{align}\label{norm:pcpl}
   \|x(k)\|^2=\|x(0)\|^2,\ & \forall\ k \in \Z
  \end{align}

  Conversely, if we define an inner product on $N(a)$ such that $\{x(k)\}$ is an orthogonal basis satisfying formula \eqref{norm:pcpl}, then Nelson theorem applies and thus the corresponding module is unitarisable.
 \item Assume that $a_1$ is not an integer but $a_2$ is a non negative integer. In this case, an integer $k$ belongs to $\cal K _0$ if and only if $k\leq a_2$. Moreover, since $a_1-a_2 \in\ \R$, we must have $a_1\in\ \R$. Then the condition becomes 
 $$\forall\ k <a_2,\ \frac{k-a_2 }{k+a_1+1} \in\ \R_{>0}.$$ Therefore, we must have $k+1+a_1<0$ for all $k<a_2$. This is true if and only if $a_1<-a_2$. Then we can express $\|x(k)\|^2$ uniquely in terms of $\|x(a_2)\|^2$, via the formula:
  \begin{align}\label{norm:HW}
   \|x(a_2-k)\|^2=\frac{k!}{\displaystyle\prod_{j=1}^k\ (j-1-a_1-a_2)}\|x(a_2)\|^2,\ & \forall\ k>0.
  \end{align}

  Conversely, if we define an inner product on $N(a)$ such that $\{x(k)\}$ is an orthogonal basis satisfying formula \eqref{norm:HW}, then Nelson theorem applies and thus the corresponding module is unitarisable.
  \item Assume that $a_2$ is not an integer but $a_1$ is a non negative integer. In this case, an integer $k$ belongs to $\cal K _0$ if and only if $k\geq -a_1$. Moreover, since $a_1-a_2 \in\ \R$, we must have $a_2\in\ \R$. Then the condition becomes 
 $$\forall\ k \geq -a_1,\ \frac{k-a_2 }{k+a_1+1} \in\ \R_{>0}.$$ Therefore, we must have $k-a_2>0$ for all $k\geq -a_1$. This is true if and only if $a_2<-a_1$. Then we can express $\|x(k)\|^2$ uniquely in terms of $\|x(-a_1)\|^2$, via the formula:
   \begin{align}\label{norm:LW}
   \|x(k-a_1)\|^2=\frac{k!}{\displaystyle\prod_{j=1}^k\ (j-1-a_1-a_2)}\|x(-a_1)\|^2,\ & \forall\ k>0.
  \end{align}
  
  Conversely, if we define an inner product on $N(a)$ such that $\{x(k)\}$ is an orthogonal basis satisfying formula \eqref{norm:LW}, then Nelson theorem applies and thus the corresponding module is unitarisable.
  \item Assume that both $a_1$ and $a_2$ are non negative integers. In this case, an integer $k$ belongs to $\cal K _0$ if and only if $-a_1\leq k \leq a_2$. Let $-a_1\leq k <a_2$ then the condition becomes 
 $$\forall\ k \in\ \Z,\ \frac{k-a_2 }{k+a_1+1} \in\ \R_{>0}.$$ This is not true, unless $a_1=a_2=0$. This choices correspond to the trivial (one-dimensional) module, which is of course unitarisable. In this case, we recovered the fact that a finite dimensional representation of a non-compact group cannot be unitary unless it is trivial.
\end{enumerate}

\subsection{Case $(II)$}

In this case, the condition becomes:
$$(a_{1}+k+1)(a_2-k)\|x(k)\|^2=-\|x(k+1)\|^2.$$ Therefore we must have $(a_{1}+k+1)(k-a_2)>0$. For an integer $k$ to belong to $\cal K _0$ it is necessary that $a_2-k<0$. Therefore we must have $a_1+k+1>0$. Let us distibguish two situations:
\begin{enumerate}[(i)]
 \item If $a_1 \not\in\ \Z$, then the condition $a_1+k+1>0$ for all $k>a_2$ is true if and only if $a_1+a_2+2>0$. In this case, we can express $\|x(k)\|^2$ via the formula:
 \begin{align}\label{norm:LW00}
  \|x(k+a_2+1)\|^2 = (k!)\displaystyle\prod_{j=1}^{k}\ (j+1+a_1+a_2)\|x(a_2+1)\|^2,\ & \forall\ k>0
 \end{align}
 
  Conversely, if we define an inner product on $N(a)$ such that $\{x(k)\}$ is an orthogonal basis satisfying formula \eqref{norm:LW00}, then Nelson theorem applies and thus the corresponding module is unitarisable.
\item If $a_1$ is a non negative integer, then an integer $k>a_2$ is in $\cal K _0$ if and only if $k+a_1\geq 0$. Hence in this case the condition is fulfilled. Then we can express $\|x(k)\|^2$ via the formula:
 \begin{subequations}\label{norm:LW0}
 
 \begin{align}
  \|x(k-a_1)\|^2 = (k!)\displaystyle\prod_{j=1}^k\ (j-1-a_1-a_2)\|x(-a_1)\|^2,\ & \forall\ k >0,\ \mbox{if } -a_1>a_2,\\
  \|x(k+a_2+1)\|^2 = (k!)\displaystyle\prod_{j=1}^{k}\ (j+1+a_1+a_2)\|x(a_2+1)\|^2,\ & \forall\ k>0,\ \mbox{if } -a_1\leq a_2.
 \end{align}
 
 \end{subequations}
 
  Conversely, if we define an inner product on $N(a)$ such that $\{x(k)\}$ is an orthogonal basis satisfying formula \eqref{norm:LW0}, then Nelson theorem applies and thus the corresponding module is unitarisable.
\end{enumerate}

\subsection{Case $(III)$}

This case is analoguous to the previous one. More specifically we have two situations:
\begin{enumerate}[(i)]
  \item If $a_2 \not\in\ \Z$, then we find the condition $a_1+a_2+2>0$. In this case, we can express $\|x(k)\|^2$ via the formula:
 \begin{align}\label{norm:LW100}
  \|x(-k-a_1-1)\|^2 = (k!)\displaystyle\prod_{j=1}^{k}\ (j+1+a_1+a_2)\|x(-a_1-1)\|^2,\ & \forall\ k>0
 \end{align}
 
  Conversely, if we define an inner product on $N(a)$ such that $\{x(k)\}$ is an orthogonal basis satisfying formula \eqref{norm:LW100}, then Nelson theorem applies and thus the corresponding module is unitarisable.
\item If $a_2$ is a non negative integer, then the unitarisability condition is fulfilled and we can express $\|x(k)\|^2$ via the formula:
 \begin{subequations}\label{norm:LW10}
 
 \begin{align}
  \|x(a_2-k)\|^2 = (k!)\displaystyle\prod_{j=1}^k\ (j-1-a_1-a_2)\|x(a_2)\|^2,\ & \forall\ k >0,\\
   & \mbox{if } -a_1>a_2,\notag \\
  \|x(-k-a_1-1)\|^2 = (k!)\displaystyle\prod_{j=1}^{k}\ (j+1+a_1+a_2)\|x(-a_1-1)\|^2,\ & \forall\ k>0,\\
   & \mbox{if } -a_1\leq a_2.\notag
 \end{align}
 
 \end{subequations}
 
  Conversely, if we define an inner product on $N(a)$ such that $\{x(k)\}$ is an orthogonal basis satisfying formula \eqref{norm:LW10}, then Nelson theorem applies and thus the corresponding module is unitarisable.
\end{enumerate}

\subsection{Case $(IV)$}

In this case, the condition becomes:
$$(a_2-k)\|x(k)\|^2=-(a_1+k+1)\|x(k+1)\|^2.$$
Furthermore, an integer $k$ belongs to $\cal K _0$ if and only if $k+a_1<0$ and $a_2-k<0$. Therefore the condition is never fulfilled. Of course, in this case the corresponding module $N(a)$ is finite dimensional ; so we know \emph{a priori} that it is not unitarisable.

\subsection{Statement}

Let us now state the final result:

\begin{theo}
 Let $a=(a_1,a_2)\in\ \C^2$. The module $N(a)$ is unitarisable if and only if $a$ is of one of the following form:
 \begin{enumerate}
 \item $a=(-1-x+iy,x+iy)$, with $x\in\ \R,\ y \in\ \R$.
 \item $a=(a_1,a_2)$, with $a_1$, $a_2$ non integer real numbers, and $-2-[a_2]<a_1<-1-[a_2]$.
 \item $a=(a_1,a_2)$, with $a_1\in\ \Z_{\geq 0}$ and $a_2 \in\ \R\setminus\Z$ such that $a_1+a_2+2<0$.
 \item $a=(a_1,a_2)$, with $a_2\in\ \Z_{\geq 0}$ and $a_1 \in\ \R\setminus\Z$ such that $a_1+a_2+2<0$.
 \item $a=(a_1,a_2)$, with $a_1\in\ \Z_{\geq 0}$ and $a_2\in\ \Z_{<0}$.
 \item $a=(a_1,a_2)$, with $a_2\in\ \Z_{\geq 0}$ and $a_1\in\ \Z_{<0}$.
 \item $a=(0,0)$.
\end{enumerate}

\end{theo}

In this classification, there are a lot of repetitions. For instance if $a_1$ and $a_2$ are not integers we have $N(a_1,a_2)=N(a_1-k,a_2+k)$, for any integer $k$. Up to isomorphism, this list reduces to the following:

\begin{enumerate}[(i)]
 \item $N(-1-x+iy,x+iy)$, $-1\leq x< 0$, $y\in\ \R_{>0}$ (Principal Series).
 \item $N(a_1,a_2)$, $-1<a_1,a_2<0$ (Complementary Series).
 \item $N(a_1,0)$, $a_1<0$ or $N(0,a_2)$, $a_2<0$ (Discrete Series and Continuations).
 \item $N(0,0)$ (Trivial Representation).
\end{enumerate}
In the sequel we denote the same way a unitarisable module and the corresponding unitary representation of the universal covering of $SU(1,1)$.

\begin{rema}
 The first proof of the classification of the unitary dual of the universal covering of $SU(1,1)$ is due to Pukanszky \cite{Pu64}. See also \cite{Sa66}. Another proof in the same spirit than our can be found in \cite{JM84}. There, J\o rgensen and Moore proved a stronger result: any simple weight module is the differential of a continuous representation of the universal covering of $SU(1,1)$ in some Hilbert space. 
\end{rema}
 
To conclude this section we collect the support and the infinitesimal character of the unitarisable modules in table \ref{tab:supp-char}.

\begin{table}[htbp]
 
\begin{center}
 
\begin{tabular}{|c|c|c|}
 \hline
  & & \\
 {\bf Modules} & {\bf Support} & {\bf Infinitesimal Character}\\
  & & \\
 \hline
 & & \\
$\pi_{-x,iy}=N(-1-x+iy,x+iy)$ & $-1-2x+2\Z$ & $-\frac{1}{4}-y^2$\\
{\Small ({\it Principal Series})} & & \\
 \hline
  & & \\
$\pi^c_{-a_{1},-a_{2}}=N(a_{1},a_{2})$ & $a_1-a_2+2\Z$ & $\left(\frac{a_1+a_2}{2}\right)\left(1+\frac{a_1+a_2}{2}\right)$\\
{\Small ({\it Complementary Series})} & & \\
 \hline
  & & \\
$\pi^+_{-a_{1}}=N(a_{1},0)$ & $a_1-2\Z_{\leq 0}$ & $\frac{a_1}{2}\left(1+\frac{a_1}{2}\right)$\\
{\Small ({\it Highest Weight})} & & \\
 \hline
  & & \\
$\pi^-_{-a_{2}}=N(0,a_{2})$ & $-a_2+2\Z_{\geq 0}$ & $\frac{a_2}{2}\left(1+\frac{a_2}{2}\right)$\\
{\Small ({\it Lowest Weight})} & & \\
 \hline 
\end{tabular}
\caption{}\label{tab:supp-char}
\end{center}

\end{table}


\section{Tensor Products : algebraic approach}


In this section we will investigate the algebraic structure of tensor products of $\germ sl (2,\C)$-modules. More precisely, we will be interested in tensor products of one of the following form :
\begin{enumerate}[(i)]
\item $N(0,b)\otimes N(a,0)$, with $a,\ b\in\ \R_{<0}$,
\item $N(-1-x+iy,x+iy)\otimes N(a,0)$, with $-1\leq x< 0$, $y\in\ \R_{>0}$, and  $a\in\ \R_{<0}$,
\item $N(a_{1},a_{2})\otimes N(a,0)$, with $-1< a_{1},a_{2}<0$, $a\in\ \R_{<0}$.
\end{enumerate}
In all cases, we denote by $V$ the tensor product. We give a basis of $V$ as follows. Let $x(k)$ be the standard basis of $N(0,b)$ (resp. $N(-1-x+iy,x+iy)$, $N(a_1,a_2)$), where $k$ belongs to $\Z_{\geq 0}$ (resp. $\Z$). Let $y(l)$ be the basis of $N(a,0)$ defined by $y(l)=x(-l)$, where $x(j)$ is the standard basis of $N(a,0)$ and $l$ belongs to $\Z_{\geq 0}$. Set $z(k,l)=x(k)\otimes y(l)$. This is a basis of $V$. Using formulae $(I)$, $(II)$, $(III)$, and $(IV)$ of section \ref{ssec:sl2}, we check that 
\begin{enumerate}[(i)]
 \item $H\cdot z(k,l)=\left(-b+a+2(k-l)\right) z(k,l).$
 \item $H\cdot z(k,l)=\left(-1-2x+a+2(k-l)\right) z(k,l).$
 \item $H\cdot z(k,l)=\left(a_{1}-a_{2}+a+2(k-l)\right) z(k,l).$
\end{enumerate}
We deduce then that $V$ is the direct sum of its weight spaces and that all its non zero weight spaces are infinite dimensional. Moreover we have $supp(V)=b_1-b_2+a+2\Z$, where $(b_1,b_2)=(0,b)$ (resp. $(-1-x+iy,x+iy)$, $(a_1,a_2)$). From \cite{Fe90}, we know that every submodule (resp. quotient) of $V$ is also the direct sum of its weight spaces. More specifically, if $W$ is a submodule of $V$, then for any $\lambda \in\ \mathfrak{h}^*$ we have $W_{\lambda}=V_{\lambda}\cap W$ and $\left(V/W\right)_{\lambda}=V_{\lambda}/\left(V_{\lambda}\cap W\right)$.

Let $\cal U _0$ denote the commutant of $\germ h $ in $\cal U (\germ g )$. Then, as an algebra, $\cal U _0$ is generated by $H$ and $FE$. In other words, a basis of $\cal U _0$ is given by the vectors $(FE)^tH^s$ for $t,\ s \in\ \Z_{\geq 0}$. Now recall the following general result:

\begin{theo}[Lemire \cite{Le69}]\label{thm:Lemire}
 Let $\germ g $ be a simple finite dimensional complex Lie algebra. Let $\germ h $ be a Cartan subalgebra of $\germ g $. Denote by $\cal U _0$ the commutant of $\germ h $ in $\cal U (\germ g )$.
 \begin{enumerate}
  \item Let $M$ be a simple weight $\germ g $-module. Then for any $\lambda \in\ \germ h ^*$, $M_{\lambda}$ is either zero or a simple $\cal U _0$-module.
  \item Let $M_0$ be a simple $\cal U _0$-module. Then up to isomorphism, there is a unique simple weight module $M$ such that $M_0$ is a weight space of $M$.
 \end{enumerate}
\end{theo}

Finally recall form proposition \ref{prop:WMsl2} that a simple weight module for $\germ sl (2,\C)$ is of degree $1$. Thus theorem \ref{thm:Lemire} implies that to construct all simple submodules (resp. quotient) of $V$ it is sufficient to understand all simple $\cal U _0$-submodule (resp. quotient) of all weight spaces of $V$. We have seen above that weight spaces of $V$ are indexed by integers. Let $n_0\in\ \Z$. Denote by $V_{n_0}$ the weight space of weight $b_1-b_2+a+2n_0$, where $(b_1,b_2)=(0,b)$ (resp. $(-1-x+iy,x+iy)$, $(a_1,a_2)$). Then a basis of this weight space is given by all the vectors $z(k,l)$ such that $k-l=n_0$. Using formulae $(I)$, $(II)$, $(III)$, and $(IV)$ of section \ref{ssec:sl2}, we see that in general we have $(FE)\cdot z(k,l)=a(k,l)z(k-1,l-1)+b(k,l)z(k,l)+c(k,l)z(k+1,l+1)$ for some complex numbers $a(k,l),\ b(k,l),\ c(k,l)$. Moreover, we always have $c(k,l)\not=0$.

Let $l_0$ be the smallest $l$ such that there exists $k$ with $k-l=n_0$. The integer $l_0$ exists since $l\in\ \Z_{\geq 0}$. More precisely, in the cases $(ii)$ and $(iii)$ we always have $l_0=0$ since $k\in\ \Z$. In the case $(i)$, $l_0=0$ when $n_0\geq 0$ and $l_0=-n_0$ when $n_0<0$, since $k\in\ \Z_{\geq 0}$. The formulae show that we always have $a(l_0+n_0,l_0)=0$. Denote by $\germ c $ the one-dimensional Lie algebra $\C (FE)$. We prove:

\begin{prop}\label{prop:VNcyclic}
 With the notations as above, we have:
 \begin{enumerate}
  \item As a $\germ c $-module, $V_{n_0}$ is cyclic, generated by $z_{n_0}:=z(l_0+n_0,l_0)$.
  \item The map $\rho_{n_0}: \cal U (\germ c ) \rightarrow V_{n_0}$ defined by $\rho_{n_0}(u)=u\cdot z_{n_0}$ is a bijection.
 \end{enumerate}
\end{prop}
\begin{proof}
Denote by $z_{n_0}(j):=z(l_0+n_0+j,l_0+j)$ for $j\in\ \Z_{\geq 0}$. Then $V_{n_0}$ has a basis given by the vectors $z_{n_0}(j)$ for $j\in\ \Z_{\geq 0}$.
 \begin{enumerate}
  \item We have $(FE)\cdot z_{n_0}=b(l_0+n_0,l_0)z_N+c(l_0+n_0,l_0)z_{n_0}(1)$, with $c(l_0+n_0,l_0)\not=0$. Thus $\frac{(FE)-b(l_0+n_0,l_0)1}{c(l_0+n_0,l_0)}\cdot z_{n_0}=z_{n_0}(1)$. Therefore $z_{n_0}(1)\in\ \cal U (\germ c )\cdot z_{n_0}$. We then prove that $z_{n_0}(j)\in\ \cal U (\germ c )\cdot z_{n_0}$ by induction on $j$, using the relation 
 $(FE)\cdot z_{n_0}(j)=a(l_0+n_0+j,l_0+j)z_{n_0}(j-1)+b(l_0+n_0+j,l_0+j)z_{n_0}(j)+c(l_0+n_0+j,l_0+j)z_{n_0}(j+1)$. This completes the first part of the proposition.
  \item The map $\rho_{n_0}$ is surjective by the first part. We prove that it is also injective. Let $u=\displaystyle\sum_{m=0}^M\ c_m(FE)^m \in\ \cal U (\germ c )$ such that $c_M\not=0$. Then using the relation $(FE)\cdot z_{n_0}(j)=a(l_0+n_0+j,l_0+j)z_{n_0}(j-1)+b(l_0+n_0+j,l_0+j)z_{n_0}(j)+c(l_0+n_0+j,l_0+j)z_{n_0}(j+1)$, we check that $u\cdot z_{n_0}=c_M\times c(l_0+n_0+M-1,l_0+M-1)z_{n_0}(M) + \displaystyle\sum_{m=0}^{M-1}\ d_mz_{n_0}(m)$ for some complex numbers $d_m$. As the vectors $z_{n_0}(j)$ are independant vectors of $V_{n_0}$, we conclude that $u\cdot z_{n_0}\not=0$. Hence $\rho_{n_0}$ is injective.
 \end{enumerate} 
\end{proof}

A consequence of proposition \ref{prop:VNcyclic} is the following:

\begin{coro}
 As a $\cal U (\germ c )$-module, $V_{n_0}$ is isomorphic to $\C[X]$, the space of polynomials in one indeterminate.
\end{coro}

Now remark that $H$ acts on $V_{n_0}$ as a scalar. Therefore $W$ is a simple $\cal U _0$-submodule (resp. quotient) of $V_{n_0}$ if and only if it is a simple $\cal U (\germ c )$-submodule (resp. quotient) of $V_{n_0}$. As $V_{n_0}=\C[X]$ as a $\cal U (\germ c )$-module, we conclude that it does not have any simple submodule and that simple quotient of $V_{n_0}$ are of the form $V_{n_0}/(FE-\chi)$ for some complex number $\chi$. Such a quotient is one-dimensional (as expected), generated by a vector $z$ satisfying $H\cdot z=(b_1-b_2+a+2n_0)z$ and $FE\cdot z=\chi z$.

Thanks to theorem \ref{thm:Lemire}, we conclude that $V$ does not have any simple submodules and that $W$ is a simple quotient of $V$ if and only if $W$ has a one-dimensional weight space generated by a vector $z$ satisfying $H\cdot z=(b_1-b_2+a+2n_0)z$ and $FE\cdot z=\chi z$ for some integer $n_0$ and some complex number $\chi$. Note that if $W$ is a simple weight $\germ sl (2,\C)$-module such that $supp(W)\subset supp(V)$ then there is an integer $n_0$ such that $b_1-b_2+a+2n_0$ is a weight of $W$. The corresponding weight space is one-dimensional as asserted by the proposition \ref{prop:WMsl2}. Let $z$ be any vector of $W$ of weight $b_1-b_2+a+2n_0$. Then the Casimir operator $\Omega$ acts on $W$ as a scalar $\chi'$ and therefore $FE$ acts on $z$ by a scalar $\chi$. As a conclusion we have proved the following:

\begin{theo}\label{thm:tens-alg}
 Every simple weight $\germ sl (2,\C)$-module $W$ such that $supp(W)\subset supp(V)$ appears as a quotient of $V$.
\end{theo}


\section{Tensor products : Hilbertian approach}


In this section, we will investigate the structure of tensor products of unitarisable $\germ su (1,1)$-modules. In what is to follow, we will set $V=N(a_{1},a_{2})\otimes N(a,0)$ where $a\in\ \R_{<0}$, and either $a_{1}=0,\ a_{2}\in\ \R_{<0}$ or $a_{1}=-1-x+iy,\ a_{2}=x+iy$ (with $-1\leq x< 0$, $y\in\ \R_{>0}$) or $-1<a_{1},a_{2}<0$. Set $s=a+a_1+a_2$. If $a_{1}=-1-x+iy,\ a_{2}=x+iy$, then $s=-1+a+2iy$, and so $\Re(s)<-1$. Otherwise, $s\in\ \R_{<0}$. We will consider several different cases:
\begin{enumerate}[(A)]
\item $a\not\in\ \Z$ and either $a_{1}=0,\ a_{2}\in\ \R_{<0}\setminus \Z_{<0}$ or $a_{1}=-1-x+iy,\ a_{2}=x+iy$ (with $-1\leq x< 0$, $y\in\ \R_{>0}$) or $-1<a_{1},a_{2}<0$.
\item $a\in\ \Z_{<0}$ and either $a_{1}=0,\ a_{2}\in\ \R_{<0}\setminus \Z_{<0}$ or $a_{1}=-1-x+iy,\ a_{2}=x+iy$ (with $-1\leq x< 0$, $y\in\ \R_{>0}$) or $-1<a_{1},a_{2}<0$.
\item $a\not\in\ \Z_{<0}$ and $a_{1}=0,\ a_{2}\in \Z_{<0}$.
\item $a\in\ \Z_{<0}$ and $a_{1}=0,\ a_{2}\in \Z_{<0}$.
\end{enumerate}

Denote by $x(k)$ the standard basis of $N(a_{1},a_{2})$ given in section \ref{sssec:Weyl}. In particular, $k\in\ \Z_{\geq 0}$ if $a_{1}=0$ and $k\in\ \Z$ otherwise. For $l\in\Z_{\geq 0}$, denote by $y(l)$ the basis of $N(a,0)$ defined by $y(l)=x(-l)$ where $x(j)$ is the standard basis of $N(a,0)$. Now a basis for $V$ is $z(k,l)=x(k)\otimes y(l)$.

Moreover the modules $N(a_{1},a_{2})$ and $N(a,0)$ have a Hilbert space structure, given by formulae \eqref{norm:cpl}, \eqref{norm:pcpl}, \eqref{norm:HW}, or \eqref{norm:LW}. Therefore we can construct a Hilbert space structure on $V$ via the formula $\langle x\otimes y,x'\otimes y'\rangle=\langle x,x'\rangle\times \langle y,y'\rangle$. Thus the completion of $V$ with respect to this Hilbert structure is 
$$\hat{V}=\left\{\displaystyle\sum_{k,l}\ u_{k,l}z(k,l)\ : \ \displaystyle\sum_{k,l}\  |u_{k,l}|^2\| z(k,l)\|^2< \infty\right\}.$$
In the sequel we shall write $V$ instead of $\hat{V}$. For future use we recall in table  \ref{tab:asymptotics} the value of the norms $\|x(k)\|^2$ and $\|y(l)\|^2$ in the various situations.

\begin{table}[htbp]
\begin{center}
\begin{tabular}{|c|c|}
\hline
 & \\
$a \not\in\ \Z_{<0}$ & $\|y(l)\|^2=\frac{l!}{\displaystyle\prod_{j=1}^l\ (j-a-1)}\|y(0)\|^2,\ l>0$ \\
 & \\
\hline
 & \\
$a \in\ \Z_{<0}$ & $\|y(l)\|^2=l!\displaystyle\prod_{j=1}^l\ (j-a-1)\|y(0)\|^2,\ l>0$ \\
 & \\
\hline
 & \\
$a_{2} \not\in\ \Z_{<0}$ & $\|x(k)\|^2=\frac{\displaystyle\prod_{j=1}^k\ (j+a_1)}{\displaystyle\prod_{j=1}^k\ (j-\bar a_2-1)}\|x(0)\|^2,\ k>0$ \\
 & \\
\hline
 & \\
$a_{1}=0,\ a_{2}\in\ \Z_{<0}$ & $\|x(k)\|^2=k!\displaystyle\prod_{j=1}^k\ (j-a_2-1)\|x(0)\|^2,\ k>0$ \\
 & \\
\hline 
\end{tabular}
\caption{}\label{tab:asymptotics}
\end{center}
\end{table}
From now on, we shall assume that $\|x(0)\|^2=1=\|y(0)\|^2$. We conclude this paragraph by giving the action of $H$, $E$ and $F$ in the four cases. In fact, the action of $H$ is the same in all cases, and is given by
$$H\cdot z(k,l)=(a_{1}-a_{2}+a+2(k-l))z(k,l).$$ Remark in particular that $Supp(V)=a_{1}-a_{2}+a+2\Z$ and that all non zero weight spaces are infinite dimensional.
The action of $E$ and $F$ is given by:
\begin{enumerate}[(A)]
\item \begin{subequations}\label{action1}
\begin{align}
E \cdot z(k,l) = & (a_{2}-k)z(k+1,l)+lz(k,l-1)\\
F\cdot z(k,l) = & (a_{1}+k)z(k-1,l)+(a-l)z(k,l+1)
\end{align}
\end{subequations}
\item \begin{subequations}\label{action2}
\begin{align}
E \cdot z(k,l) = & (a_{2}-k)z(k+1,l)+l(a-l+1)z(k,l-1)\\
F\cdot z(k,l) = & (a_{1}+k)z(k-1,l)+z(k,l+1)
\end{align}
\end{subequations}
\item \begin{subequations}\label{action3}
\begin{align}
E \cdot z(k,l) = & z(k+1,l)+lz(k,l-1)\\
F\cdot z(k,l) = & k(a_{2}-k+1)z(k-1,l)+(a-l)z(k,l+1)
\end{align}
\end{subequations}
\item \begin{subequations}\label{action4}
\begin{align}
E \cdot z(k,l) = & z(k+1,l)+l(a-l+1)z(k,l-1)\\
F\cdot z(k,l) = & k(a_{2-k+1)}z(k-1,l)+z(k,l+1)
\end{align}
\end{subequations}
\end{enumerate}


\subsection{Highest and lowest weight modules}


In this section, we investigate which highest or lowest weight modules are submodules of $V$. Remark that $z(k,l)$ and $z(k',l')$ have the same weight if $k-l=k'-l'$. Let $n_0 \in\ \Z$. Assume first that $n_0\geq 0$. Consider a vector of the form 
$$v=\displaystyle\sum_{l\geq 0}\ u_lz(l+n_0,l).$$
We want to determine $u_l$ such that $E\cdot v=0$ and $\displaystyle\sum_{l\geq 0}\ |u_l|^2\|z(l+n_0,l)\|^2<\infty$. From equations \eqref{action1}, \eqref{action2}, \eqref{action3}, \eqref{action4}, we see that in general 
$$E \cdot z(l+n_0,l)=a(l+n_0)z(l+n_0+1,l)+b(l)z(l+n_0,l-1),$$ where $a(l+n_0)\not=0$, $b(0)=0$, and $b(l)\not=0$ for positive $l$. Now the equation $E\cdot v=0$ gives:
$$\displaystyle\sum_{l\geq 0}\ \left(u_la(l+n_0)+u_{l+1}b(l+1)\right)z(l+n_0+1,l)=0.$$
Therefore, we must have 
$$\forall\ l\geq 0,\ u_la(l+n_0)+u_{l+1}b(l+1)=0.$$
Hence, 
$$u_l=(-1)^l\times \frac{\displaystyle\prod_{j=1}^l\ a(j+n_0-1)}{\displaystyle\prod_{j=1}^l\ b(j)}u_0.$$
We will assume now that $u_0=1$. To check the convergence condition we give the asymptotic behavior of $|u_l|^2$ in the four cases:
\begin{enumerate}[(A)]
 \item $|u_l|^2 \sim l^{n_0-1-a_2}$
 \item $|u_l|^2 \sim \frac{l^{n_0+a-a_2}}{l!}$
 \item $|u_l|^2 \sim \frac{1}{l!}$
 \item $|u_l|^2 \sim \frac{l^{a+1}}{(l!)^2}$
\end{enumerate}
Now, recall that $\|z(l+n_0,l)\|^2=\|x(l+n_0)\|^2\|y(l)\|^2$. Using the asymptotic behavior given in table \ref{tab:asymptotics}, we conclude that in all cases we have $|u_l|^2\|z(l+n_0,l)\|^2\sim l^{a+a_1-a_2+2n_0}$. Thus, the convergence condition holds if and only if $2n_0<-1-a-a_1+a_2$.

Assume now that $n_0<0$. Then if $a_1\not=0$, the vector $z(l+n_0)$ exists for all non negative $l$. In this case, the above computation still holds. Hence we find a highest weight vector of weight $a_1-a_2+a+2n_0$ in $V$ if and only if $2n_0<-1-a-a_1+a_2$. If $a_1=0$, the vector $z(l+n_0)$ exists only for $l+n_0\geq 0$. In this case, the equation $E\cdot \displaystyle\sum_{l\geq -n_0}\ u_lz(l+n_0,l)=0$ gives $u_{-n_0}=0$ and by induction $u_l=0$ for all $l$.

We can now summarize our results in the following:

\begin{prop}\label{prop:discHW}
If a simple highest weight module $N(\lambda,0)$, of highest weight $\lambda$, is a Hilbert submodule of $V$, then $\lambda=a_1+a-a_2+2n_0$ for some integer $n_0$. Conversely, 
 \begin{enumerate}
  \item Assume that $a_1=0$. Then the simple highest module $N(a-a_2+2n_0,0)$, of highest weight $a-a_2+2n_0$, is a Hilbert submodule of $V$ if and only if $0\leq 2n_0<-1-a+a_2$.
  \item Assume that $a_1\not=0$. Then the simple highest module $N(a_1+a-a_2+2n_0,0)$, of highest weight $a_1+a-a_2+2n_0$, is a Hilbert submodule of $V$ if and only if $2n_0<-1-a-a_1+a_2$.
 \end{enumerate}
\end{prop}

Let us now turn to lowest weight modules. First assume that $n_0\geq 1$. We want to determine $u_l$ such that $F\cdot \displaystyle\sum_{l\geq 0}\ u_lz(l+n_0,l)=0$ and $\displaystyle\sum_{l\geq 0}\ |u_l|^2\|z(l+n_0,l)\|^2<\infty$. As above, we write in general 
$$F\cdot z(l+n_0,l)=a'(l+n_0)z(l+n_0-1,l)+b'(l)z(l+n_0,l+1).$$ We remark that we have $u_0=0$ (since $a'(n_0)\not=0$)  and by induction $u_l=0$. This still holds if $n_0<1$ and $a_1\not=0$.

Assume then that $n_0<1$ and $a_1=0$. We want to determine $u_l$ such that $F\cdot \displaystyle\sum_{l\geq -n_0}\ u_lz(l+n_0,l)=0$ and $\displaystyle\sum_{l\geq -n_0}\ |u_l|^2\|z(l+n_0,l)\|^2<\infty$. As above, we write in general 
$$F\cdot z(l+n_0,l)=a'(l+n_0)z(l+n_0-1,l)+b'(l)z(l+n_0,l+1).$$ Now we have $a'(0)=0$. As above we write 
$$0=F\cdot \displaystyle\sum_{l\geq -n_0}\ u_lz(l+n_0,l)=\displaystyle\sum_{l\geq -n_0}\ (u_{l+1}a'(l+n_0+1)+u_lb'(l))z(l+n_0,l+1).$$ We deduce from that the expression of $u_l$, that is
$$u_l=(-1)^l\times \frac{\displaystyle\prod_{j=1}^l\ b'(j-1)}{\displaystyle\prod_{j=1}^l\ a'(j+n_0)}u_0.$$
We then find the asymptotic behavior of $|u_l|^2$ from which we conclude that the asymptotic behavior of $|u_l|^2\|x(l+n_0,l)\|^2$ is $l^{a_2-a-2n_0}$. Thus we have proved the following:

\begin{prop}\label{prop:discLW}
 If a simple lowest weight module $N(0,-\lambda)$, of lowest weight $\lambda$, is a Hilbert submodule of $V$, then $\lambda=a_1+a-a_2+2n_0$ for some integer $n_0$. Conversely, 
 \begin{enumerate}
  \item Assume that $a_1=0$. Then the simple lowest module $N(0,-a+a_2-2n_0)$, of lowest weight $a-a_2+2n_0$, is a Hilbert submodule of $V$ if and only if $1+a_2-a<2n_0\leq 0$.
  \item Assume that $a_1\not=0$. Then $V$ has no Hilbert submodule isomorphic to the simple lowest module $N(0,-a_1-a+a_2-2n_0)$, of lowest weight $a_1+a-a_2+2n_0$, for any $n_0$.
 \end{enumerate}
\end{prop}


\subsection{Principal and complementary series}


In this section, we investigate which modules from the principal or the complementary series are submodules of $V$. Recall that the support of such a module $M$ is of the form $b+2\Z$. As we have $supp(V)=a_1-a_2+a+2\Z$, we can assume that $b=a_{1}-a_{2}+a$, that is $M=N(b_{1},b_{2})$ with $b_{1}-b_{2}=a_{1}-a_{2}+a$ and either $b_{1}=-1-x'+iy',\ b_{2}=x'+iy'$ or $-1<b_{1},b_{2}<0$. Let $v$ denote the weight vector of $N(b_{1},b_{2})$ having weight $b_{1}-b_{2}$. Then from the action of $E$ and $F$ given in section \ref{ssec:sl2}, we find that $FE\cdot v=\xi v$ with 
\begin{subequations}\label{cond:pcpl-cpl}
\begin{align}
\xi \leq -\left(\frac{1+a+a_{1}-a_{2}}{2}\right)^2, & \ \mbox{if } b_{1}=-1-x'+iy',\ b_{2}=x'+iy',\\
-\left(\frac{1+a+a_{1}-a_{2}}{2}\right)^2 < \xi <  -&\left(\frac{a+a_{1}-a_{2}}{2}\right)\left(\frac{a+a_{1}-a_{2}+2}{2}\right),\\
 & \ \mbox{if } -1<b_{1},b_{2}<0.\notag
\end{align}
\end{subequations}

Now remark that the vector $z(k,l)$ has weight $a_{1}-a_{2}+a$ if and only if $k=l$. Therefore we are looking for a vector $$v=\displaystyle\sum_{n\geq 0}\ u_{n}z(n,n)$$ such that $FE\cdot v=\xi v$ (with $\xi$ satisfying one of the conditions \eqref{cond:pcpl-cpl}) and $$\displaystyle\sum_{n\geq 0}\ |u_{n}|^2\|z(n,n)\|^2<\infty.$$ Conversely if the vector $v$ satisfy both these conditions it is easy to check that $v$ generates a simple submodule of $V$.

Using equations \eqref{action1}, \eqref{action2}, \eqref{action3}, and \eqref{action4}, we compute $FE\cdot v$ in the four cases and write it in the form 
$$FE\cdot v=\displaystyle\sum_{n\geq 0}\ v_{n}z(n,n),$$ for some sequence $v_n$ completely determined by the $u_k$. Then we can identify the coefficients of $z(n,n)$ in $FE\cdot v$ and in $\xi v$. We obtain:
\begin{enumerate}[(A)]
\item \begin{subequations}\label{eq:diffA}
\begin{align}
 & (a_{1}+1)u_{1}+(a_{2}(a_{1}+1))u_{0} = \xi u_{0},\\
 & (n+2)(n+2+a_{1})u_{n+2}\\
 & +\left((a_{2}-n-1)(a_{1}+n+2)+(n+1)(a-n)\right)u_{n+1}\notag \\
 & +(a-n)(a_{2}-n)u_{n} = \xi u_{n+1},\ \forall\ n\geq 0.\notag
\end{align}
\end{subequations}
\item \begin{subequations}
\begin{align}
 & a(a_{1}+1)u_{1}+(a_{2}(a_{1}+1))u_{0} = \xi u_{0},\\
 & (n+2)(n+2+a_{1})(a-n-1)u_{n+2}\\
 & +\left((a_{2}-n-1)(a_{1}+n+2)+(n+1)(a-n)\right)u_{n+1}\notag\\
 & +(a_{2}-n)u_{n} = \xi u_{n+1},\ \forall\ n\geq 0.\notag
\end{align}
\end{subequations}
\item \begin{subequations}
\begin{align}
 & a_2u_{1}+a_{2}u_{0} = \xi u_{0},\\
 & (n+2)^2(a_{2}-n-1)u_{n+2}\\
 & +\left((a_{2}-n-1)(n+2)+(n+1)(a-n)\right)u_{n+1}\notag \\
 & +(a-n)u_{n} = \xi u_{n+1},\ \forall\ n\geq 0.\notag
\end{align}
\end{subequations}
\item \begin{subequations}
\begin{align}
 & aa_2u_{1}+a_{2}u_{0} = \xi u_{0},\\
 & (n+2)^2(a-n-1)(a_{2}-n-1)u_{n+2}\\
 & +\left((a_{2}-n-1)(n+2)+(n+1)(a-n)\right)u_{n+1}\notag \\
 & +u_{n} = \xi u_{n+1},\ \forall\ n\geq 0.\notag
\end{align}
\end{subequations}
\end{enumerate}

In the first case, we see using table \ref{tab:asymptotics} that $\|z(n,n)\|^2 \sim (n+1)^{2+\Re(s)}$. Therefore, the sequence $u_n$ belongs to the Hilbert space $V$ if and only if $\displaystyle\sum_{n\geq 0}\ |u_{n}|^2(n+1)^{2+\Re(s)}<\infty$.

Now we consider the following change of variable:
\begin{enumerate}
\item[(B)] $v_{0}=u_{0}$ and $v_{n}=\displaystyle\prod_{j=1}^n\ (a+1-j)\times u_{n},\ \forall\ n>0$.
\item[(C)] $v_{0}=u_{0}$ and $v_{n}=\displaystyle\prod_{j=1}^n\ (a_{2}+1-j)\times u_{n},\ \forall\ n>0$.
\item[(D)] $v_{0}=u_{0}$ and $v_{n}=\displaystyle\prod_{j=1}^n\ (a+1-j)(a_{2}+1-j)\times u_{n},\ \forall\ n>0$.
\end{enumerate}

Then it is easily check that the sequence $v_{n}$ satisfy equations \eqref{eq:diffA}. Moreover, the condition $\displaystyle\sum_{n\geq 0}\ |u_{n}|^2\|z(n,n)\|^2<\infty$ is then equivalent to the condition $\displaystyle\sum_{n\geq 0}\ |v_{n}|^2(n+1)^{2+\Re(s)}<\infty$, which is the condition satisfied by the sequence $u_{n}$ in case $(A)$.

Set $\mu = \xi -a_{2}(1+a+a_{1})$ and $p=aa_{2}$. Note that in all cases $\mu+\left(\frac{1+s}{2}\right)^2$ is a real number which satisfies:
$$\mu+\left(\frac{1+s}{2}\right)^2\leq 0,\ \mbox{if } x \mbox{ generates a module from the principal series},$$
$$0<\mu+\left(\frac{1+s}{2}\right)^2<\frac{1}{4},\ \mbox{if } x \mbox{ generates a module from the complementary series}.$$ From the above discussion, we are left with the following two equations:

\begin{subequations}\label{eq:diffA1}
\begin{align}
& (a_{1}+1)u_{1}=(p+\mu)u_{0}, \\
& (n+2)(n+2+a_{1})u_{n+2}\label{eqg:diffA}\\
& +\left(s+2-\mu - (n+2)(n+2+a_{1}) - (n-a)(n-a_{2})\right)u_{n+1}\notag \\
& +(n-a)(n-a_{2})u_{n}=0.\notag
\end{align}
\end{subequations}

It is clear that this difference equation has a unique solution for a given $u_{0}$. In the following, we shall assume without loss of generality that $u_{0}=1$. To check wether 
$$\displaystyle\sum_{n\geq 0}\ |u_{n}|^2n^{2+\Re(s)}<\infty$$ holds, we need to understand the asymptotic behavior of the unique solution. We will use two different approaches.

\subsubsection{Asymptotics using a discrete derivative}

Equation \eqref{eqg:diffA} has two independent fundamental solutions. We will denote them by $u_n^1$ and $u_n^2$. Then our sequence $u_n$ satisfying equations \eqref{eq:diffA1} is an unknown linear combination of these solutions.

Define an operator $D$ (discrete derivative) by the formula $D(u_{n})=u_{n+1}-u_{n}$. Then we can rewrite equation \eqref{eqg:diffA} with $D$ as follows
\begin{equation}\label{eq:diffD}
(n+2)(n+2+a_{1})D^2(u_{n})+\left(6+s-p+2a_{1}-\mu +n(4+s)\right)D(u_{n})+(s+2-\mu)u_{n}=0.
\end{equation}
Now we will use a discrete version of the local analysis of differential equation (see \cite{BO78}). For  equation \eqref{eq:diffD}, the point $\infty$ is regular-singular (\cite[section 5.2]{BO78}). Therefore, we know from the discrete version of Fuchs theorem (\cite[section 5.2]{BO78}) that the fundamental solutions of the difference equation have asymptotics of the form $n^{\alpha}\displaystyle\sum_{n\geq 0}\ A_{k}n^{-k}$ and $n^{\beta}\displaystyle\sum_{n\geq 0}\ B_{k}n^{-k}$ or $n^{\alpha}\ln (n)\displaystyle\sum_{n\geq 0}\ A_{k}n^{-k}$, for some complex numbers $\alpha$ and $\beta$, and where $A_0$ and $B_0$ are not zero. To find $\alpha$ (and $\beta$), we write $u^1_n=A_0n^{\alpha}+A_1n^{\alpha-1}+A_2n^{\alpha-2}+o(n^{\alpha-2})$ for large $n$. Evaluating now equation \eqref{eqg:diffA} gives
$(\alpha^2+(3+s)\alpha+s+2-\mu)A_0+o(\frac{1}{n^2})=0$, from which we conclude that 
$\alpha^2+(3+s)\alpha+s+2-\mu=0$. Thus we find:
$$\alpha=\frac{-s-3}{2}\pm \sqrt{\mu+\left(\frac{1+s}{2}\right)^2}.$$

%
%

Now, if $v$ generates a module from the principal series, we have seen that $\mu+\left(\frac{1+s}{2}\right)^2\leq 0$. First, if $\mu+\left(\frac{1+s}{2}\right)^2< 0$, then the square-norm of both fundamental solutions is equivalent to $n^{2\Re(\alpha)}= n^{-\Re(s)-3}$. But any solution is a linear combination of these fundamental solutions. Hence the square-norm of every solution is equivalent to $n^{-\Re(s)-3}$. If $\mu+\left(\frac{1+s}{2}\right)^2 = 0$, then the square-norm of the fundamental solutions are equivalent to $n^{-\Re(s)-3}$ or to $n^{-\Re(s)-3}\ln^2 (n)$. Hence the square-norm of every solution is equivalent to $n^{-\Re(s)-3}$ or to $n^{-\Re(s)-3}\ln^2 (n)$. Thus if $v$ generates a module from the principal series, we have $ |u_{n}|^2n^{2+\Re(s)} \sim n^{-1}$ or $|u_{n}|^2n^{2+\Re(s)} \sim n^{-1}\ln^2 (n)$. Hence, the sequence $u_{n}$ is not in the Hilbert space $V$. Thus we have proved the 

\begin{prop}\label{prop:discPS}
 Let $b_1=-1-x'+iy'$ and $b_2=x'+iy'$, with $y'\not=0$. Then the simple weight module $N(b_1,b_2)$ is a never a Hilbert submodule of $V$.
\end{prop}

Note however that the principal series whose support is $a_1-a_2+a+2\Z$ is almost contained in the Hilbert space, in the sense that it is contained in 
$$V_{\epsilon}:=\left\{\sum_{k,l}\ u_{k,l}z(k,l)\ : \ \sum_{k,l}\ |u_{k,l}|^2\|z(k,l)\|^2(k^2+l^2)^{-\epsilon}<\infty \right\},$$ for any positive $\epsilon$.

On the other hand, if $v$ generates a module from the complementary series, we have seen that $0<\mu+\left(\frac{1+s}{2}\right)^2<\frac{1}{4}$. Hence we find that 
$|u^1_{n}|^2n^{2+\Re(s)} \sim n^{-1-\sqrt{\mu+\left(\frac{1+s}{2}\right)^2}}$ and $|u^2_{n}|^2n^{2+\Re(s)} \sim n^{-1+\sqrt{\mu+\left(\frac{1+s}{2}\right)^2}}$. As our solution $u_n$ is an unknown linear combination of $u_n^1$ and $u_n^2$, we cannot conclude.

\subsubsection{Asymptotics using a differential equation}

As we have seen, we need an other approach to deal with complementary series. From now on, we assume that 
$0<\mu+\left(\frac{1+s}{2}\right)^2<\frac{1}{4}$. For $-1<t<1$, set $S(t)=\displaystyle\sum_{n\geq 0}\ u_{n}t^n$. Then, the sequence $u_{n}$ satisfies the equations \eqref{eq:diffA1} if and only if $S(t)$ is a solution of the following differential equation:
\begin{equation}
t(1-t)S''(t)+(1+a_{1}-(1+a_{1}-s)t)S'(t)-\left(p+\frac{\mu}{1-t}\right)S(t)=0
\end{equation}
Moreover, we must have $S(0)=1$ and $S'(0)=\frac{p+\mu}{1+a_{1}}$. The unique solution to this Cauchy problem is the function
$$S(t)=(1-t)^r\!\phantom{l}_{2}F_{1}(r-a,r-a_{2};1+a_{1};t),$$ where $r=\frac{1+s}{2}-\sqrt{\mu+\left(\frac{1+s}{2}\right)^2}$ and $\!\phantom{l}_{2}F_{1}(r-a,r-a_{2};1+a_{1};t)$ is the corresponding hypergeometric function:
$$\!\phantom{l}_{2}F_{1}(r-a,r-a_{2};1+a_{1};t)=\displaystyle\sum_{n\geq 0}\ \frac{(r-a)_{n}(r-a_{2})_{n}}{(1+a_{1})_{n}}\frac{t^n}{n!}.$$
Here $(b)_{n}$ is the Pochhammer symbol, that is 
$$(b)_{0}=1, \quad (b)_{n}=\displaystyle\prod_{j=1}^n\ (b-1+j),\forall\ n>0.$$ As $1+a_1\not\in\ \Z_{\leq 0}$, the function $S(z)$ is well-defined on the (open) unit disc $D$ and is holomorphic. Remark that $a$ and $a_2$ play a symmetric role in the definition of $S$. In the sequel, we shall write $a_{(2)}$ to refer either to $a$ or $a_2$. Before going further, we need to collect several facts about the hypergeometric function. We refer the reader to \cite{Ba35} or \cite{AAR99}.

\begin{lemm}[Gauss Theorem]\label{lem:Gauss}
 Let $\alpha,\ \beta,\ \gamma \in\ \C$ such that $\alpha\not\in\ \Z_{\leq 0}$, $\beta\not\in\ \Z_{\leq 0}$,  and $\gamma\not\in\ \Z_{\leq 0}$. Then:
 \begin{enumerate}
 \item The function $z\mapsto \!\phantom{l}_2F_1(\alpha,\beta;\gamma;z)$ is holomorphic in $D$.
 \item If $\Re(\gamma-\alpha-\beta)>0$, then the function $z\mapsto \!\phantom{l}_2F_1(\alpha,\beta;\gamma;z)$ is continuous in $\bar{D}$.
 \item If $\Re(\gamma-\alpha-\beta)<0$, then there is a non zero constant $C$ such that $$\!\phantom{l}_2F_1(\alpha,\beta;\gamma;z) \sim C(1-z)^{\gamma-\alpha-\beta},\ \mbox{when}\ z\rightarrow 1.$$
 \item If $\gamma-\alpha-\beta=0$, then there is a non zero constant $C$ such that $$\!\phantom{l}_2F_1(\alpha,\beta;\gamma;z) \sim C\log (1-z),\ \mbox{when}\ z\rightarrow 1.$$
\end{enumerate}
\end{lemm}
\begin{proof}
 The first and the second assertions are theorem 2.1.2 of \cite{AAR99}. The last two assertions are theorem 2.1.3 of \cite{AAR99}.
\end{proof}

\begin{lemm}\label{lem:hyperpol}
 Assume that $r-a$ or $r-a_2$ is a non positive integer.
 \begin{enumerate}
 \item Then $\!\phantom{l}_2F_1(r-a,r-a_2;1+a_1;z)$ is polynomial, and is therefore holomorphic in $\C$.
 \item We cannot have $r-a\in\ \Z_{\leq 0}$ and $r-a_2\in\ \Z_{\leq 0}$ unless $r=a=a_2$.
 \item We have $\!\phantom{l}_2F_1(r-a,r-a_2;1+a_1;1)\not=0$.
 \item If $\Re(s)\geq -2$, then $r-a\in\ \Z_{\leq 0}$ (resp. $r-a_2\in\ \Z_{\leq 0}$) implies $r=a$ (resp. $r=a_2$) and therefore $\!\phantom{l}_2F_1(r-a,r-a_2;1+a_1;x)=1$.
\end{enumerate}
\end{lemm}
\begin{proof}
The first assertion follows from the definition of the Pochhammer symbol.

 Assume that $r-a=-n$ and $r-a_2=-m$. Then $2r-a-a_2+n+m=0$, that is $a_1+1+n+m-2\sqrt{\mu+\left(\frac{s+1}{2}\right)^2}=0$. Therefore we should have $0<1+a_1+m+n<1$. But we have $1+a_1=-x+iy$ or $1+a_1>0$. In the first case the equality $a_1+1+n+m-2\sqrt{\mu+\left(\frac{s+1}{2}\right)^2}=0$ can only hold when $y=0$. Then in both cases $0<1+a_1\leq 1$. Therefore the equality can only hold for $n+m=0$, i.e. $n=m=0$. This prove the second part of the lemma.
 
 The third assertion is a consequence of Chu-Vandermonde theorem \cite[cor 2.2.3]{AAR99}. Indeed, assume for instance that $r-a=-n$. Then $1+a_1-r+a_2=1+s-a-r=1+s-n-2r=2\sqrt{\mu+\left(\frac{s+1}{2}\right)^2}-n$. This last quantity is never an integer. But Chu-Vandermonde theorem implies that $\!\phantom{l}_2F_1(r-a,r-a_2;1+a_1;1)=\frac{(1+a_1-r+a_2)_n}{(1+a_1)_n}$. Therefore, this is not zero.
 
 Finally, assume that $r-a_{(2)}=-n$ for $n\in\ \Z_{>0}$. Then we must have $0<n+\frac{\Re(s)+1}{2}-\Re(a_{(2)})<\frac{1}{2}$. By our hypothesis $\frac{\Re{s}+1}{2}\geq -\frac{1}{2}$. But we have $\Re(a_{(2)})\leq 0$. Thus such a condition never holds.
\end{proof}

\begin{lemm}\label{lem:hyper}
Assume that $r-a$ and $r-a_2$ are not non positive integers.
\begin{enumerate}[(i)]
\item The hypergeometric function is well-defined and continuous on the close unit disc $\bar D$.
\item We have $$\!\phantom{l}_{2}F_{1}(r-a,r-a_{2};1+a_{1};1)=\frac{\Gamma(1+a_{1})\Gamma(1+s-2r)}{\Gamma(1+a_{1}+a-r)\Gamma(1+a_{1}+a_{2}-r)}.$$
\item For $z\in D$, we have $$\!\phantom{l}_{2}F_{1}'(r-a,r-a_{2};1+a_{1};z)=\frac{(r-a)(r-a_{2})}{1+a_{1}}\!\phantom{l}_{2}F_{1}(r-a+1,r-a_{2}+1;2+a_{1};z).$$
\item The derivative of $\!\phantom{l}_{2}F_{1}(r-a,r-a_{2};1+a_{1};z)$ is well-defined and continuous on the domain $\bar D \setminus\{1\}$.
\item When $z\rightarrow 1$, there is a non zero constant $C$ such that  $$\!\phantom{l}_{2}F_{1}'(r-a,r-a_{2};1+a_{1};z) \sim C(1-z)^{s-2r}.$$
\end{enumerate}
\end{lemm}
\begin{proof}
Remark that $1+a_{1}-(r-a)-(r-a_{2})=1+s-2r=2\sqrt{\mu+\left(\frac{1+s}{2}\right)^2}$ and that $0<\mu+\left(\frac{1+s}{2}\right)^2<\frac{1}{4}$. Now $(i)$ is a consequence of \cite[thm 2.1.2]{AAR99}, $(ii)$ is Gauss theorem \cite[thm 2.2.2]{AAR99}, $(iii)$ is \cite[eq. (2.5.1)]{AAR99}, $(iv)$ is a consequence of $(iii)$ and \cite[thm 2.1.2]{AAR99}, and $(v)$ is a consequence of $(iii)$ and \cite[thm 2.1.3]{AAR99}.
\end{proof}

\begin{coro}\label{cor:hyper}
Assume that $r-a$ and $r-a_2$ are not non positive integers. If $1+a_{1}+a-r\not\in\ \Z_{\leq 0}$ and $1+a_{1}+a_{2}-r\not\in\ \Z_{\leq 0}$, then $\!\phantom{l}_{2}F_{1}(r-a,r-a_{2};1+a_{1};1)\not=0$.
\end{coro}

\begin{lemm}\label{lem:hyperN}
Assume that $r-a$ and $r-a_2$ are not non positive integers. Assume also that $1+a_{1}+a-r\in\ \Z_{\leq 0}$ or $1+a_{1}+a_{2}-r\in\ \Z_{\leq 0}$.
 \begin{enumerate}
 \item We cannot have $1+a_{1}+a-r\in\ \Z_{\leq 0}$ and $1+a_{1}+a_{2}-r\in\ \Z_{\leq 0}$.
 \item If $1+a_{1}+a-r=-n$, we have $\!\phantom{l}_2F_1(1+a_1+n,r-a_2;1+a_1;z)=(1-z)^{a_2-r-n}P_n(z)$, where $P_n(z)$ is a polynomial of degree $n$. If $1+a_{1}+a_2-r=-n$, we have $\!\phantom{l}_2F_1(1+a_1+n,r-a;1+a_1;z)=(1-z)^{a-r-n}Q_n(z)$, where $Q_n(z)$ is a polynomial of degree $n$.
 \item With the notations as above, we have $P_0=1=Q_0$ and, for $n>0$, $P_n(1)=\frac{(r-a_2)(r-a_2+1)\cdots (r-a_2+n-1)}{(1+a_1)(2+a_1)\cdots (n+a_1)}\not=0$ and $Q_n(1)=\frac{(r-a)(r-a+1)\cdots (r-a+n-1)}{(1+a_1)(2+a_1)\cdots (n+a_1)}\not=0$.
 \item If $\Re(s)\geq -2$, then $1+a_{1}+a-r=-n$ (resp. $1+a_{1}+a_2-r=-n$) implies that $\Re(s)<-1$ and $n=0$, and therefore $\!\phantom{l}_2F_1(1+a_1,r-a_2;1+a_1;z)=(1-z)^{a_2-r}$ (resp.  $\!\phantom{l}_2F_1(1+a_1,r-a;1+a_1;z)=(1-z)^{a-r}$).
\end{enumerate}
\end{lemm}
\begin{proof}
Assume that $1+a_{1}+a-r=-n$ and $1+a_{1}+a_{2}-r=-m$. Then we have $1+s+1+a_1-2r+n+m=0$, that is $1+a_1+n+m+2\sqrt{\mu+\left(\frac{1+s}{2}\right)^2}=0$. So we must have $1+a_1\in\ \R$ and $-1<1+a_1+n+m<0$. As $1+a_1>0$, this is not possible.

The second assertion of the lemma is Euler's theorem \cite[thm 2.2.5]{AAR99}. The value of $P_n$ at $z=1$ is Chu-Vandermonde's theorem \cite[cor 2.2.3]{AAR99}.

To prove the fourth part of the lemma, assume for instance that $1+a_1-r+a+n=0$. Therefore we have  $\frac{1+\Re(s)}{2}+n-\Re(a_2)+\sqrt{\mu+\left(\frac{s+1}{2}\right)^2}=0$. So we must have $0<-\frac{1+\Re(s)}{2}-n+\Re(a_2)<\frac{1}{2}$. As we have $\Re(a_2)< 0$ and $\Re(s)\geq -2$ by hypothesis, such an equality can only hold if $n=0$. Moreover we must have $\Re(s)>-1$. The proof when $1+a_1+a_2-r=-n$ is analogous.
\end{proof}

From now on, we shall set  $F(z)=\!\phantom{l}_2F_1(r-a,r-a_2;1+a_1;z)$. For future use, we shall compute some asymptotics.

\begin{lemm}\label{lem:asym}
\begin{enumerate}
 \item Let $d\in\ \R$. Then, when $te^{i\theta}\rightarrow 1$, there is a non zero constant $C$ such that $$\left(1-te^{i\theta}\right)^d \sim C \left|1-te^{i\theta}\right|^d.$$
 \item Assume that $r-a$ and $r-a_2$ are not non positive integers. Assume also that $F(1)\not=0$. Then, when $te^{i\theta}\rightarrow 1$, there is a non zero constant $C$ such that 
$$F'\left(te^{i\theta}\right) \sim C\left|1-te^{i\theta}\right|^{s-2r}.$$
\item Let $d\in\ \C$. Then there is a non zero constant such that 
$$|(1-z)^d|^2\sim C|1-z|^{2\Re(d)},\ \mbox{when } z\rightarrow 1.$$
\end{enumerate}
\end{lemm}
\begin{proof}
 The first part of the lemma follows from the equality $$\left(1-te^{i\theta}\right)^d=\left|1-te^{i\theta}\right|^d\times \exp{\left(2id\arctan \left(\frac{-t\sin \theta}{1-t\cos \theta+\sqrt{1+t^2-2t\cos\theta}}\right)\right)}.$$ The second assertion follows from the first part together with lemma \ref{lem:hyper} (note that $s-2r\in\ \R$). The last part of the lemma follows from the equality 
 $$(1-z)^d=|1-z|^{\Re(d)}\times \exp \left( i\Im(d)\ln |1-z|+2id\arctan\left(\frac{-\Im(z)}{1-\Re(z)+|1-z|}\right)\right).$$
\end{proof}

Let $d^*\theta$ denote the measure on $[0,2\pi]$ such that $\int_{0}^{2\pi}\ d^*\theta=1$. 

\begin{lemm}\label{lem:theta-int}
 Let $a\in\ D$. Then we have 
 $$\int_0^{2\pi}\ \left(1+a^2-2a\cos (\theta)\right)^{\nu}d^*\theta = \!\phantom{l}_2F_1\left(-\nu,-\nu;1;a^2\right).$$
\end{lemm}
\begin{proof}
 This is equation $(3.665(2))$ p.427 of \cite{GR94}.
\end{proof}

\begin{coro}\label{cor:theta-int}
 Let $\nu$ be such that $\nu\not\in\ \Z_{\leq 0}$ and $\Re(1+2\nu)<0$. When $t\rightarrow 1$, there is a non zero constant $C$ such that  
 $$\int_0^{2\pi}\ \left|1-te^{i\theta}\right|^{2\nu}d^*\theta \sim C\left(1-t^2\right)^{1+2\nu}.$$
\end{coro}
\begin{proof}
 Remark that $\left|1-te^{i\theta}\right|^{2\nu}=\left(1+a^2-2a\cos (\theta)\right)^{\nu}$. Then apply lemma \ref{lem:theta-int} and lemma \ref{lem:Gauss}.
\end{proof}

\begin{lemm}\label{lem:asymS}
Assume that $r-a\not\in\ \Z_{\leq 0}$, $r-a_{2}\not\in\ \Z_{\leq 0}$, $1+a_{1}+a-r\not\in\ \Z_{\leq 0}$, and $1+a_{1}+a_{2}-r\not\in\ \Z_{\leq 0}$. Then 
 \begin{enumerate}
 \item We have $$S'(z)=(1-z)^rF'(z)+(-r)(1-z)^{r-1}F(z).$$
 \item We have $$S(z)\overline{S'(z)z}=\left|(1-z)^r\right|^{2}F(z)\overline{F'(z)z} +(-\bar{r})\left|(1-z)^{r-1}\right|^{2}|F(z)|^2(\bar{z}-|z|^2).$$ When $z\rightarrow 1$, there are non zero constants $C$ and $C'$ such that  
 $$\left|(1-z)^r\right|^{2}F(z)\overline{F'(z)z} \sim C|1-z|^{\Re(s)}$$ and 
 $$\left|(1-z)^{r-1}\right|^{2}|F(z)|^2\bar{z}(1-z) \sim C'|1-z|^{2\left(\Re(r)-\frac{1}{2}\right)}.$$ Furthermore, if $r\not=0$, then there is a non zero constant $C''$ such that 
 $$S(z)\overline{S'(z)z} \sim C''|1-z|^{2\left(\Re(r)-\frac{1}{2}\right)}.$$
 \item We have \begin{multline*}
 |S'(z)|^2=\left|(1-z)^r\right|^{2}|F'(z)|^2+|r|^2\left|(1-z)^{r-1}\right|^{2}|F(z)|^2\\
 +2\left|(1-z)^{r-1}\right|^{2}\left((-r)(1-z)F(z)\overline{F'(z)} +(-\bar{r})(1-\bar{z})F'(z)\overline{F(z)}\right).
 \end{multline*}
 When $z\rightarrow 1$, there are non zero constants $C$, $C'$, and $C''$ such that  
$$\left|(1-z)^r\right|^{2}|F'(z)|^2 \sim C|1-z|^{2(\Re(s)-\Re(r))},$$
$$\left|(1-z)^{r-1}\right|^{2}|F(z)|^2 \sim C'|1-z|^{2(\Re(r)-1)},$$
$$\left|(1-z)^{r-1}\right|^{2}(1-z)F(z)\overline{F'(z)} \sim C''|1-z|^{\Re(s)-1}.$$
Furthermore, if $r\not=0$, then there is a non zero constant $C'''$ such that 
$$ |S'(z)|^2 \sim C'''|1-z|^{2(\Re(r)-1)}.$$
\end{enumerate}
\end{lemm}
\begin{proof}
The equalities are clear. The equivalents are consequences of the fact that $F(1)\not=0$ by corollary \ref{cor:hyper} and of those equivalents in lemma \ref{lem:asym}.
\end{proof}

Now, we need to transform our infinitesimal condition $\displaystyle\sum_{n\geq 0}\ |u_n|^2(n+1)^{2+\Re(s)}<\infty$ into an equivalent condition satisfied by the function $S(z)=(1-z)^rF(z)$. We denote by $\cal O (D)$ the set of holomorphic functions of the unit disc. Let $dvol(z)$ denote the measure on $D$ such that $\int_{D}\ dvol(z)=1$. First, remark that $\Re(s)<0$. We will distinguish several cases.

$\blacktriangleright$ Assume that $\Re(s)<-2$ and consider the following space:
$$\cal H _s:=\left\{f\in\ \cal O (D)\ :\ \displaystyle\int_D\ |f(z)|^2\left(1-|z|^2\right)^{-3-\Re(s)}dvol(z)<\infty \right\}.$$
Then it is clear that for all non negative $n$ the function $f_n(z)=z^n$ belongs to the Hilbert space $\cal H _s$. Moreover we have $\langle f_n,\ f_m\rangle =0$ if $n\not=m$, and $\langle f_n,\ f_n\rangle \sim (n+1)^{2+\Re(s)}$ for large $n$. Therefore, $\langle S,\ S\rangle =\displaystyle\sum_{n\geq 0}\ |u_n|^2\langle f_n,\ f_n\rangle .$ Hence $u_n \in\ V$ if and only if $S \in\ \cal H _s.$

$\blacktriangleright$ Now assume that $\Re(s)=-2$. Then the infinitesimal conndition is $\displaystyle\sum_{n\geq 0}\ |u_n|^2<\infty$. Therefore we need to consider the Hilbert space:
$$\cal H _s:=\left\{f\in \cal O (D)\ :\ \lim_{\rho\rightarrow 1}\ \int_0^{2\pi}\ \left|f(\rho e^{i\theta})\right|^2d^*\theta <\infty\right\}.$$
It is well-known that the functions $f_n(z)=z^n$ give an orthonormal basis of $\cal H _s$. Hence $u_n \in\ V$ if and only if $S\in\ \cal H _s$.

$\blacktriangleright$ Assume now that $-2<\Re(s)<-1$ and consider the Hilbert space:
$$
 \cal H _s:=\left\{f\in\ \cal O (D)\ :\ \int_D\ \left(f(z)\overline{f'(z)z}\right)\left(1-|z|^2\right)^{-2-\Re(s)}dvol(z) + |f(0)|^2 <\infty \right\}
$$
It is easily check that the functions $f_n(z)=z^n$ belong to $\cal H _s$, are mutually orthogonal and satisfy $\|f_{n}\|_s^2 \sim (n+1)^{2+\Re(s)}$ for large $n$. Thus, the sequence $u_n$ belongs to $V$ if and only if $S\in\ \cal H _s$.

$\blacktriangleright$ Assume now that $\Re(s)=-1$, in which case $s\in\ \R$, and consider the Hilbert space:
$$
 \cal H _s:=\left\{f\in\ \cal O (D)\ :\ \lim_{\rho\rightarrow 1}\ \int_0^{2\pi}\ \left(f(\rho e^{i\theta})\overline{f'(\rho e^{i\theta})e^{i\theta}}\right)d^*\theta +|f(0)|^2<\infty \right\}
$$
It is easily check that the functions $f_n(z)=z^n$ belong to $\cal H _s$, are mutually orthogonal and satisfy $\|f_{n}\|_s^2 = n$ for large $n$. Thus, the sequence $u_n$ belongs to $V$ if and only if $S\in\ \cal H _s$.

$\blacktriangleright$ Assume now that $-1<\Re(s)<0$, in which case $s\in\ \R$, and consider the Hilbert space:
$$
 \cal H _s:=\left\{f\in\ \cal O (D)\ :\ \int_D\ |f'(z)|^2\left(1-|z|^2\right)^{-1-\Re(s)}dvol(z) + |f(0)|^2<\infty \right\}$$

It is easily check that the functions $f_n(z)=z^n$ belong to $\cal H _s$, are mutually orthogonal and satisfy $\|f_{n}\|_s^2 \sim (n+1)^{2+\Re(s)}$ for large $n$. Thus, the sequence $u_n$ belongs to $V$ if and only if $S\in\ \cal H _s$.

Now we are in position to prove the following three propositions:

\begin{prop}
Assume that $r-a\in\ \Z_{\leq 0}$ or $r-a_2\in\ \Z_{\leq 0}$. Then the function $S(z)$ does not belong to $\cal H _s$.
\end{prop}
\begin{proof}
$\blacktriangleright$ Assume first that $r=a$ or $r=a_2$ (in particular $r\not=0$). Then $F(z)=1$ and therefore $S(z)=(1-z)^{r}$.
 \begin{enumerate}
 \item If $\Re(s)<-2$, then $S(z)\in\ \cal H _s$ if and only if 
 $$\int_D\ \left|(1-z)^r\right|^2\left(1-|z|^2\right)^{-3-\Re(s)}dvol(z)<\infty.$$ This function is integrable if and only if it is integrable near $z=1$. Using lemma \ref{lem:asym}, this is equivalent to the following condition:
 $$\int_{0}^1\ \left(\int_{0}^{2\pi}\ \left|1-te^{i\theta}\right|^{2\Re(r)}d^*\theta\right)\left(1-t^2\right)^{-3-\Re(s)}tdt<\infty.$$
Now, remark that $\Re(-r)>0$ and $\Re(1+2r)<0$. So, from corollary \ref{cor:theta-int}, we know that there is a non zero constant $C$ such that 
 $$\int_0^{2\pi}\ \left|1-te^{i\theta}\right|^{2\Re(r)}d^*\theta \sim C\left(1-t^2\right)^{1+2\Re(r)}.$$ So our condition becomes 
 $$\int_0^1\ \left(1-t^2\right)^{2\Re(r)-2-\Re(s)}tdt<\infty.$$
 But $-2+2\Re(r)-\Re(s)=-1-2\sqrt{\mu+\left(\frac{s+1}{2}\right)^2}<-1$, therefore our function is not integrable.
 \item If $\Re(s)=-2$, then $S(z)\in\ \cal H _s$ if and only if 
 $$\lim_{\rho\rightarrow 1}\ \int_0^{2\pi}\ \left|(1-\rho e^{i\theta})^r\right|^{2}d^*\theta<\infty,$$ or also via lemma \ref{lem:asym} if and only if 
  $$\lim_{\rho\rightarrow 1}\ \int_0^{2\pi}\ \left|1-\rho e^{i\theta}\right|^{2\Re(r)}d^*\theta<\infty.$$ Now, remark that $\Re(-r)>0$ and $\Re(1+2r)<0$. So, from corollary \ref{cor:theta-int}, we know that there is a non zero constant $C$ such that 
 $$\int_0^{2\pi}\ \left|1-\rho e^{i\theta}\right|^{2\Re(r)}d^*\theta \sim C\left(1-\rho^2\right)^{1+2\Re(r)}.$$ Therefore, the above limit is infinite.
 \item If $-2<\Re(s)<-1$, then $S(z)\in\ \cal H _s$ if and only if 
 $$\int_D\ -\bar{r}\left|(1-z)^{r-1}\right|^{2}(1-z)\bar{z}\left(1-|z|^2\right)^{-2-\Re(s)}dvol(z)<\infty.$$
 This function is integrable if and only if it is integrable near $z=1$. Thanks to lemma \ref{lem:asym}, there is some non zero constant $C$ such that 
 $$-\bar{r}\left|(1-z)^{r-1}\right|^{2}(1-z)\bar{z}\left(1-|z|^2\right)^{-2-\Re(s)}\sim C|1-z|^{2\left(\Re(r)-\frac{1}{2}\right)}\left(1-|z|^2\right)^{-2-\Re(s)}.$$ Its integral on $D$ is 
 $$\int_0^1\ \left(\int_0^{2\pi}\ \left|1-te^{i\theta}\right|^{2\left(\Re(r)-\frac{1}{2}\right)}d^*\theta\right) \left(1-t^2\right)^{-2-\Re(s)}tdt.$$
 Now remark that $\Re\left(\frac{1}{2}-r\right)>0$ and $\Re(2r)<0$. Then from corollary \ref{cor:theta-int}, we know that there is a non zero constant $C$ such that 
 $$\int_0^{2\pi}\ \left|1-te^{i\theta}\right|^{2\left(\Re(r)-\frac{1}{2}\right)}d^*\theta \sim C\left(1-t^2\right)^{2\Re(r)}.$$ So our condition becomes 
 $$\int_0^1\ \left(1-t^2\right)^{2\Re(r)-2-\Re(s)}tdt<\infty.$$
 But $-2+2\Re(r)-\Re{s}=-1-2\sqrt{\mu+\left(\frac{s+1}{2}\right)^2}<-1$, and therefore the function is not integrable.
 \item If $\Re(s)=-1$, then $S(z)\in\ \cal H _s$ if and only if 
 $$\lim_{\rho\rightarrow 1}\ \int_0^{2\pi}\ -\bar{r}\left|(1-\rho e^{i\theta})^{r-1}\right|^{2}(1-\rho e^{i\theta})\rho e^{-i\theta}d^*\theta<\infty.$$ Using lemma \ref{lem:asym}, we see that this limit is finite if and only if 
 $$\lim_{\rho\rightarrow 1}\ \int_0^{2\pi}\ \left|1-\rho e^{i\theta}\right|^{2\left(\Re(r)-\frac{1}{2}\right)}d^*\theta<\infty.$$
 Now remark that $\Re\left(\frac{1}{2}-r\right)>0$ and $\Re(2r)<0$. Thus from corollary \ref{cor:theta-int}, we know that there is a non zero constant $C$ such that 
 $$\int_0^{2\pi}\ \left|1-\rho e^{i\theta}\right|^{2\left(\Re(r)-\frac{1}{2}\right)}d^*\theta \sim C\left(1-\rho^2\right)^{2\Re(r)}.$$ So our condition becomes 
 $$\lim_{\rho\rightarrow 1}\ \left(1-\rho^2\right)^{2\Re(r)}<\infty.$$
 Therefore the limit is not finite.
 \item If $-1<\Re(s)<0$, then $S(z)\in\ \cal H _s$ if and only if 
 $$\int_0^1\ \left(\int_0^{2\pi}\ |r|^2\left|(1-te^{i\theta})^{r-1}\right|^{2}d^*\theta\right) \left(1-t^2\right)^{-1-\Re(s)}tdt<\infty.$$ Using lemma \ref{lem:asym}, this is equivalent to the condition:
$$\int_0^1\ \left(\int_0^{2\pi}\ |r|^2\left|1-te^{i\theta}\right|^{2(\Re(r)-1)}d^*\theta\right) \left(1-t^2\right)^{-1-\Re(s)}tdt<\infty.$$ Now, remark that $\Re(1-r)>0$ and $\Re(2r-1)<0$. Hence from corollary \ref{cor:theta-int}, we know that there is a non zero constant $C$ such that 
 $$\int_0^{2\pi}\ \left|1-te^{i\theta}\right|^{2(\Re(r)-1)}d^*\theta \sim \left(1-t^2\right)^{2\Re(r)-1}.$$ So our condition becomes 
 $$\int_0^1\ \left(1-t^2\right)^{2\Re(r)-2-\Re(s)}tdt<\infty.$$ 
 But $-2+2\Re(r)-\Re{s}=-1-2\sqrt{\mu+\left(\frac{s+1}{2}\right)^2}<-1$, and therefore this function is not integrable.
\end{enumerate}
$\blacktriangleright$ Assume now that $r-a=-n$ or $r-a_2=-n$ for some positive integer $n$. Then from lemma \ref{lem:hyperpol}, we know that necessarily $\Re(s)<-2$. We also know that $F(x)$ is polynomial (of degree $n$) and that $F(1)\not=0$. Now, $S(z)\in\ \cal H _s$ if and only if 
$$\int_D\ \left|(1-z)^r\right|^{2}|F(z)|^2\left(1-|z|^2\right)^{-3-\Re(s)}dvol(z)<\infty.$$ This function is integrable if and only if it is integrable near $z=1$. But then as $F(1)\not=0$ and using lemma \ref{lem:asym}, there is some non zero constant $C$ such that $\left|(1-z)^r\right|^{2}|F(z)|^2\left(1-|z|^2\right)^{-3-\Re(s)}\sim C \left|1-z\right|^{2\Re(r)}\left(1-|z|^2\right)^{-3-\Re(s)}$ when $z\rightarrow 1$. Hence we are left with the previous situation.
\end{proof}

\begin{prop}
Assume that $r-a\not\in\ \Z_{\leq 0}$ and $r-a_2\not\in\ \Z_{\leq 0}$. Assume also that $1+a_1+a-r\in\ \Z_{\leq 0}$ or $1+a_1+a_2-r\in\ \Z_{\leq 0}$. Then the function $S(z)$ belongs to $\cal H _s$.
\end{prop}
\begin{proof}
 By lemma \ref{lem:hyperN}, we have $F(1)=0$. More precisely, there is a polynomial $P(z)$ of degree $n$ such that $P(1)\not=0$ and $F(z)=(1-z)^{a_{(2)}-r-n}P(z)$ (recall that $a_{(2)}$ denotes either $a$ or $a_2$). Thus $S(z)=(1-z)^{a_{(2)}-n}P(z)$.
 
$\blacktriangleright$ Assume first that $n=0$, that is $1+s-a_{(2)}-r=0$. Then lemma \ref{lem:hyperN} implies that $P=1$ and that $\Re(s)>-1$.
\begin{enumerate}
 \item If $-2<\Re(s)<-1$, then $S\in\ \cal H _s$ if and only if 
 $$\int_D\ -\overline{a_{(2)}}\left|(1-z)^{a_{(2)}-1}\right|^2(1-z)\bar{z}\left(1-|z|^2\right)^{-2-\Re(s)}dvol(z)<\infty.$$ This function is integrable if and only if it is integrable near $z=1$. Thanks to lemma \ref{lem:asym}, there is a non zero constant $C$ such that  
 $$\left|1-z^{a_{(2)}-1}\right|^2(1-z)\bar{z}\left(1-|z|^2\right)^{-2-\Re(s)}\sim C\left|1-z\right|^{2\left(\Re(a_{(2)})-\frac{1}{2}\right)}\left(1-|z|^2\right)^{-2-\Re(s)}.$$ So our condition becomes 
 $$\int_0^1\ \left(\int_0^{2\pi}\ \left|1-te^{i\theta}\right|^{2\left(\Re(a_{(2)})-\frac{1}{2}\right)}d^*\theta\right)\left(1-t^2\right)^{-2-\Re(s)}tdt<\infty.$$ Now remark that $\Re(2a_{(2)})<0$. Then from corollary \ref{cor:theta-int}, we know that there is a non zero constant $C$ such that 
 $$\int_0^{2\pi}\ \left|1-te^{i\theta}\right|^{2\left(\Re(a_{(2)})-\frac{1}{2}\right)}d^*\theta \sim C\left(1-t^2\right)^{2\Re(a_{(2)})}.$$ Thus our condition is now 
 $$\int_0^1\ \left(1-t^2\right)^{2\Re(a_{(2)})-2-\Re(s)}tdt<\infty.$$
  As $2\Re(a_{(2)})-2-\Re(s)=2+2\Re(s)-2\Re(r)-2-\Re(s)=-1+2\sqrt{\mu+\left(\frac{s+1}{2}\right)^2}>-1$, we conclude that this function is integrable.
 \item If $\Re(s)=-2$, then $S\in\ \cal H _s$ if and only if 
 $$\lim_{\rho\rightarrow 1}\ \int_0^{2\pi}\ \left|\left(1-\rho e^{i\theta}\right)^{a_{(2)}}\right|^{2}d^*\theta<\infty.$$ Using lemma \ref{lem:asym}, this is equivalent to 
 $$\lim_{\rho\rightarrow 1}\ \int_0^{2\pi}\ \left|1-\rho e^{i\theta}\right|^{2\Re(a_{(2)})}d^*\theta<\infty.$$ Now remark that $-\Re(a_{(2)})>0$ and $\Re(1+2a_{(2)})=2\sqrt{\mu+\left(\frac{s+1}{2}\right)^2}>0$. From lemma \ref{lem:theta-int}, we know that 
 $$\int_0^{2\pi}\ \left|1-\rho e^{i\theta}\right|^{2\Re(a_{(2)})}d^*\theta=\!\phantom{l}_2F_1\left(-\Re(a_{(2)}),-\Re(a_{(2)});1;\rho^2\right).$$ So our condition becomes 
 $$\lim_{\rho\rightarrow 1}\ \!\phantom{l}_2F_1\left(-\Re(a_{(2)}),-\Re(a_{(2)});1;\rho^2\right)<\infty.$$  From Gauss theorem (lemma \ref{lem:Gauss}) we know that  $$\rho \mapsto \!\phantom{l}_2F_1\left(-\Re(a_{(2)}),-\Re(a_{(2)});1;\rho^2\right)$$ is continuous on $[0,1]$, and hence has a limit when $\rho \rightarrow 1$.
 \item If $\Re(s)<-2$, then $S\in\ \cal H _s$ if and only if 
 $$\int_0^1\ \left(\int_0^{2\pi}\ \left|\left(1-te^{i\theta}\right)^{\Re(a_{(2)})}\right|^{2}d^*\theta\right)\left(1-t^2\right)^{-3-\Re(s)}tdt < \infty.$$ 
 Using lemma \ref{lem:asym}, this is equivalent to 
 $$\int_0^1\ \left(\int_0^{2\pi}\ \left|1-te^{i\theta}\right|^{2\Re(a_{(2)})}d^*\theta\right)\left(1-t^2\right)^{-3-\Re(s)}tdt < \infty.$$ 
Note that we always have $\Re(-a_{(2)})>0$. If $\Re(1+2a_{(2)})<0$, then from corollary \ref{cor:theta-int}, we know that there is a non zero constant $C$ such that 
 $$\int_0^{2\pi}\ \left|1-te^{i\theta}\right|^{2\Re(a_{(2)})}d^*\theta \sim C\left(1-t^2\right)^{1+2\Re(a_{(2)})}.$$ So our condition becomes 
 $$\int_0^1\ \left(1-t^2\right)^{2\Re(a_{(2)})-2-\Re(s)}tdt<\infty.$$ But $-2+2\Re(a_{(2)})-\Re{s}=-1+2\sqrt{\mu+\left(\frac{s+1}{2}\right)^2}>-1$, and therefore the function is integrable. If $\Re(1+2a_{(2)})=0$ then lemma \ref{lem:theta-int} together with lemma \ref{lem:Gauss} imply that there is a non zero constant $C$ such that 
 $$\int_0^{2\pi}\ \left|1-te^{i\theta}\right|^{2\Re(a_{(2)})}d^*\theta \sim C\log\left(1-t^2\right).$$
  So our condition becomes 
  $$\int_0^1\ \log\left(1-t^2\right)\left(1-t^2\right)^{-3-\Re(s)}tdt<\infty.$$ This function is integrable since $\Re(s)<-2$. If $\Re(1+2a_{(2)})>0$, then lemma \ref{lem:theta-int} together with lemma \ref{lem:Gauss} imply that $t\mapsto \int_0^{2\pi}\ \left|1-te^{i\theta}\right|^{2\Re(a_{(2)})}d^*\theta$ is continuous on $[0,1]$, and therefore the function $$\left(\int_0^{2\pi}\ \left|1-te^{i\theta}\right|^{2\Re(a_{(2)})}d^*\theta\right)\left(1-t^2\right)^{-3-\Re(s)}t$$ is integrable on $[0,1[$.
 \end{enumerate}
$\blacktriangleright$ Assume now that $n>0$. So lemma \ref{lem:hyperN} implies that $\Re(s)>-2$. Then $S\in\ \cal H _s$ if and only if 
 $$\int_D\ \left|(1-z)^{a_{(2)}-n}\right|^{2}|P(z)|^2\left(1-|z|^2\right)^{-3-\Re(s)}dvol(z) < \infty.$$ This function is integrable if and only if it is integrable near $z=1$. As $P(1)\not=0$ and using lemma \ref{lem:asym}, this function is integrable near $1$ if and only if 
 $$\int_0^1\ \left(\int_0^{2\pi}\ \left|1-te^{i\theta}\right|^{2(\Re(a_{(2)})-n)}d^*\theta\right)\left(1-t^2\right)^{-3-\Re(s)}tdt < \infty.$$ Remark that $\Re(n-a_{(2)})>0$ and $\Re(1+2a_{(2)}-2n)<0$. Thus from corollary \ref{cor:theta-int}, we know that there is a non zero constant $C$ such that 
 $$\int_0^{2\pi}\ \left|1-te^{i\theta}\right|^{2(\Re(a_{(2)})-n)}d^*\theta \sim C\left(1-t^2\right)^{1+2\Re(a_{(2))}-2n}.$$ So our condition becomes 
 $$\int_0^1\ \left(1-t^2\right)^{2\Re(a_{(2))}-2n-2-\Re(s)}tdt<\infty.$$
 But $-2+2\Re(a_{(2)})-2n-\Re(s)=-1+2\sqrt{\mu+\left(\frac{s+1}{2}\right)^2}>-1$, and therefore the function is integrable.
\end{proof}

\begin{prop}
Assume that $r-a\not\in\ \Z_{\leq 0}$, $r-a_2\not\in\ \Z_{\leq 0}$, $1+a_1+a-r\not\in\ \Z_{\leq 0}$ and  $1+a_1+a_2-r\not\in\ \Z_{\leq 0}$. Then the function $S(z)$ belongs to $\cal H _s$ if and only if $s$ is real, $-1<s<0$, and $r=0$.
\end{prop}
\begin{proof}
Thanks to corollary \ref{cor:hyper}, we know that $F(1)\not=0$.
 \begin{enumerate}
 \item If $\Re(s)<-2$, then $S\in\ \cal H _s$ if and only if 
 $$\int_D\ \left|(1-z)^r\right|^{2}|F(z)|^2\left(1-|z|^2\right)^{-3-\Re(s)}dvol(z) < \infty.$$ It is clear that this function is integrable if and only if it is integrable near $z=1$. But near $z=1$ we have $|F(z)|^2 \sim |F(1)|^2$ since $F$ is continuous by lemma \ref{lem:hyper}. So using lemma \ref{lem:asym}, our function is integrable near $z=1$ if and only if 
 $$\int_0^1\ \left(\int_0^{2\pi}\ \left|1-te^{i\theta}\right|^{2\Re(r)}d^*\theta\right)\left(1-t^2\right)^{-3-\Re(s)}tdt < \infty.$$
 Remark that $\Re(-r)>0$ and $\Re(2r+1)<0$. So by corollary \ref{cor:theta-int} there is a non zero constant $C$ such that 
 $$\int_0^{2\pi}\ \left|1-te^{i\theta}\right|^{2\Re(r)}d^*\theta \sim C\times \left(1-t^2\right)^{1+2\Re(r)}.$$ So $S\in\ \cal H _s$ if and only if 
 $$\int_0^1\ \left(1-t^2\right)^{2\Re(r)-2-\Re(s)}tdt < \infty.$$
  But $-2+2\Re(r)-\Re(s)=-1-2\sqrt{\mu+\left(\frac{s+1}{2}\right)^2}<-1$. Therefore the function is not integrable, that is $S\not\in\ \cal H _s$.
 \item If $\Re(s)=-2$, then $S\in\ \cal H _s$ if and only if 
 $$\lim_{\rho\rightarrow 1}\ \int_0^{2\pi}\ \left|\left(1-\rho e^{i\theta}\right)^{\Re(r)}\right|^{2}\left|F(\rho e^{i\theta})\right|^2d^*\theta <\infty.$$
 Using lemma \ref{lem:asym}, this is equivalent to 
 $$\lim_{\rho\rightarrow 1}\ \int_0^{2\pi}\ \left|1-\rho e^{i\theta}\right|^{2\Re(r)}\left|F(\rho e^{i\theta})\right|^2d^*\theta <\infty.$$ As above this is true if and only if 
 $$\lim_{\rho\rightarrow 1}\ \int_0^{2\pi}\ \left|1-\rho e^{i\theta}\right|^{2\Re(r)}d^*\theta<\infty.$$
 We have $\Re(1+2r)<0$. So by corollary \ref{cor:theta-int}, the limit is not finite. Hence $S\not\in\ \cal H _s$.
 \item If $-2<\Re(s)<-1$, then $S\in\ \cal H _s$ if and only if 
 $$\int_D\ S(z)S'(z)\bar{z}\left(1-|z|^2\right)^{-2-\Re(s)}dvol(z) < \infty.$$ This function is integrable if and only if it is integrable near $z=1$. Now remark that $2\Re(r)<1+\Re(s)<0$. Thus, according to lemma \ref{lem:asymS}, our function is integrable near $z=1$ if and only if 
 $$\int_D\ |1-z|^{2\left(\Re(r)-\frac{1}{2}\right)}\left(1-|z|^2\right)^{-2-\Re(s)}dvol(z) < \infty,$$ if and only if 
 $$\int_0^1\ \left(\int_0^{2\pi}\ \left|1-te^{i\theta}\right|^{2\left(\Re(r)-\frac{1}{2}\right)}d^*\theta\right)\left(1-t^2\right)^{-2-\Re(s)}tdt<\infty.$$
As $\frac{1}{2}-\Re(r)>0$ and $\Re(r)<0$, corollary \ref{cor:theta-int} implies that there is a non zero constant $C$ such that 
 $$\int_0^{2\pi}\ \left|1-te^{i\theta}\right|^{2\left(\Re(r)-\frac{1}{2}\right)}d^*\theta \sim C\left(1-t^2\right)^{2\Re(r)}.$$ So our condition becomes 
  $$\int_0^1\ \left(1-t^2\right)^{2\Re(r)-2-\Re(s)}tdt<\infty.$$
  But we have $2\Re(r)-2-\Re(s)=-1-2\sqrt{\mu+\left(\frac{s+1}{2}\right)^2}<-1$. Hence $S\not\in\ \cal H _s.$
 \item If $\Re(s)=-1$, then $S\in\ \cal H _s$ if and only if 
 $$\lim_{\rho\rightarrow 1}\ \int_0^{2\pi}\ S\left(\rho e^{i\theta}\right)S'\left(\rho e^{-i\theta}\right)e^{-i\theta}d^*\theta\ < \infty.$$ This case is analogous to the previous one, and left to the reader.
 \item Finally, assume that $-1<\Re(s)<0$. Then $s$ and $r$ are real. We have $S\in\ \cal H _s$ if and only if 
 $$\int_D\ \left|S'(z)\right|^2\left(1-|z|^2\right)^{-1-s}dvol(z) < \infty.$$ This is true if and only if the function is integrable near $z=1$.
 
 Suppose first that $r\not=0$. Then, according to lemma \ref{lem:asymS}, the function is integrable near $z=1$ if and only if 
  $$\int_D\ |1-z|^{2(r-1)}\left(1-|z|^2\right)^{-1-s}dvol(z) < \infty,$$ if and only if 
  $$\int_0^1\ \left(\int_0^{2\pi}\ \left|1-te^{i\theta}\right|^{2(r-1)}d^*\theta\right)\left(1-t^2\right)^{-1-s}tdt < \infty.$$
  Now remark that $1-r>0$ and $2r-1<0$. Thus from corollary \ref{cor:theta-int}, there is a non zero constant $C$ such that 
  $$\int_0^{2\pi}\ \left|1-te^{i\theta}\right|^{2(r-1)}d^*\theta \sim C\left(1-t^2\right)^{2r-1}.$$
Therefore, our condition becomes 
  $$\int_0^1\ \left(1-t^2\right)^{2r-2-s}tdt<\infty.$$ But we have $2r-2-s=-1-2\sqrt{\mu+\left(\frac{s+1}{2}\right)^2}<-1$. Hence $S\not\in\ \cal H _s.$
 
 Suppose now that $r=0$. Then, according to lemma \ref{lem:asymS}, the function is integrable near $z=1$ if and only if 
 $$\int_D\ |1-z|^{2s}\left(1-|z|^2\right)^{-1-s}dvol(z) < \infty,$$ if and only if 
  $$\int_0^1\ \left(\int_0^{2\pi}\ \left|1-te^{i\theta}\right|^{2s}d^*\theta\right)\left(1-t^2\right)^{-1-s}tdt < \infty.$$
  From lemma \ref{lem:theta-int}, we have 
  $$\int_0^{2\pi}\ \left|1-te^{i\theta}\right|^{2s}d^*\theta=\!\phantom{l}_2F_1\left(-s,-s;1;t^2\right).$$
  Thus our condition is now
  $$\int_0^1\ \!\phantom{l}_2F_1\left(-s,-s;1;t^2\right)\left(1-t^2\right)^{-1-s}tdt < \infty.$$
 Note that we always have $-s>0$. If $2s+1<0$, then lemma \ref{lem:Gauss} imply that 
 $$\!\phantom{l}_2F_1\left(-s,-s;1;t^2\right)\left(1-t^2\right)^{-1-s}t \sim \left(1-t^2\right)^{s},$$ which is integrable since $s>-1$. If $2s+1=0$, then lemma \ref{lem:Gauss} imply that 
 $$\!\phantom{l}_2F_1\left(-s,-s;1;t^2\right)\left(1-t^2\right)^{-1-s}t \sim \log\left(1-t^2\right)\left(1-t^2\right)^{-1-s},$$ which is integrable since $-1-s>-1$. If $2s+1>0$, then lemma \ref{lem:Gauss} imply that the function $t\mapsto \!\phantom{l}_2F_1\left(-s,-s;1;t^2\right)$ is continuous on $[0,1]$, and therefore  $$\int_0^1 \!\phantom{l}_2F_1\left(-s,-s;1;t^2\right)\left(1-t^2\right)^{-1-s}tdt < \infty,$$ since $-1-s>-1$. Consequently, when $r=0$, we always have $S\in\ \cal H _s$.
\end{enumerate}
\end{proof}

Let us now state the consequence of these propositions:

\begin{prop}\label{prop:discCS}
A simple weight module $N(b_1,b_2)$ in the complementary series is a Hilbert submodule of $V$ if and only if 
 \begin{itemize}
 \item[$\blacktriangleright$] Either $a_1,\ a_2 \in\ \R$, $-1<a_1+a_2+a<0$, and $N(b_1,b_2) \cong N(a+a_1,a_2)$.
 \item[$\blacktriangleright$] Or $-1<a_1,\ a_2<0$, $-2<a_1+a_2-a<-1$, and $N(b_1,b_2) \cong N(a_1,a_2-a)$.
\end{itemize}
Moreover, in the first case the submodule $N(a+a_1,a_2)$ is generated by the vector 
$$w(0)=\displaystyle\sum_{n\geq 0}\ \frac{(-a)_n(-a_2)_n}{(1+a_1)_n}\frac{z(n,n)}{n!}.$$ In the second case, the submodule $N(a_1,a_2-a+2n)$ is generated by the vector 
$$w(0)=\displaystyle\sum_{n\geq 0}\ \frac{(-a)_n}{n!}z(n,n).$$ 
\end{prop}
\begin{proof}
 Let $v$ be the standard basis vector of $N(b_1,b_2) \subset V$, of weight $b_1-b_2$. Then it is straightforward to check that $FE\cdot v=b_2(b_1+1)v$. Moreover, as we already mentionned, we can assume that $b_1-b_2=a+a_1-a_2$. In the above notations, we have $\xi=b_2(b_1+1)$. On the other hand, we have expressed $\xi$ has $\xi=\mu + a_{2}(1+a+a_{1})$. From the above three propositions, we conclude that a vector $w$ of weight $a+a_1-a_2$, such that $FE\cdot w=\xi w$ generates a submodule of $V$ if and only either $s$ is real, $-1<s<0$ and $r=0$ or there is a non negative integer $n$ such that $r=1+a_1+a+n$ or $r=1+a_1+a_2+n$.
 
 In the first case, $s$ real implies that $a_1,\ a_2$ are real ; $-1<s<0$ and $r=0$ implies $\mu=0$. Therefore, we have $b_1-b_2=a+a_1-a_2$ and $b_2(b_1+1)=a_{2}(1+a+a_{1})$. Hence, up to isomorphism, $N(b_1,b_2)=N(a+a_1,a_2)$.
 
 In the second case, if $r=1+a_1+a+n=1+s-a_2+n$, then we have $2\sqrt{\mu+\left(\frac{s+1}{2}\right)^2}=2a_2-2n-1-s$. Therefore, we must have $2a_2-2n-1-s\in\ \R$ and $0<2a_2-2n-1-s<1$. The first condition is always fulfilled. The second condition reads $1+2n<a_2-a_1-a<2+2n$. But then we have also $\mu=(n-a_2)(n+1+a+a_1)$, which implies that $\xi=n(1+n+a+a_1-a_2)$. Therefore, we have $b_1-b_2=a+a_1-a_2$ and $b_2(b_1+1)=n(1+n+a+a_1-a_2)$. The solutions of this system are $b_1=n+a+a_1-a_2,\ b_2=n$ or $b_1=-1-n,\ b_2=a_2-a_1-a-n-1$. In both cases, the corresponding module is a highest weight module, and therefore does not belong to the complementary series.
 
 If $r=1+a_1+a_2+n=1+s-a+n$, then we have 
 $2\sqrt{\mu+\left(\frac{s+1}{2}\right)^2}=2a-2n-1-s$. Therefore, we must have $2a-2n-1-s\in\ \R$ and $0<2a-2n-1-s<1$. But $2a-2n-1-s\in\ \R$ implies that $a_1,\ a_2 \in\ \R$. Now the second condition reads 
 $-2-2n<a_1+a_2-a<-1-2n$. But then we also have $\mu=(n-a)(n-a+1+s)$, which implies that $\xi=(n+a_2-a)(1+n+a_1)$. Therefore, we have $b_1-b_2=a+a_1-a_2$ and $b_2(b_1+1)=(n+a_2-a)(1+n+a_1)$. Consequently, we have $b_1=n+a_1$ and $b_2=n+a_2-a$ or $b_1=a-n-a_2-1$ and $b_2=-1-n-a_1$. If $a_1=0$, then the corresponding module is a lowest weight module, and therefore does not belong to the complementary series. If $a_1\not=0$, then $-1<a_1,\ a_2<0$. However, in this case, $a_2-a>a_2>-1$. Therefore the condition $-2-2n<a_1+a_2-a<-1-2n$ can only hold if $n=0$, which gives the asserted condition. Then the submodule is isomorphic to $N(a_1,a_2-a)$ or to $N(a-a_2-1,-1-a_1)$, which turn to be isomorphic.
\end{proof}

Let us now state a final result about the discrete spectrum of the tensor products.

\begin{theo}\label{thm:disc-spec}
 \begin{enumerate}
 \item Let $a,b<0$. Then the discrete spectrum of the Hilbert tensor product $N(0,a)\otimes N(b,0)$ is 
\begin{align*}
 N(a,b),&\ \mbox{if}\ -1<a+b<0,\\
\displaystyle\bigoplus_{0\leq 2n<a-b-1} N(b-a+2n,0),&\ \mbox{if}\ 0<a-b-1,\\
\displaystyle\bigoplus_{a-b+1<2n\leq 0} N(0,a-b-2n),&\ \mbox{if}\ a-b+1<0.
  \end{align*}
\item Let $-1\leq x<0$, $y\in\ \R\setminus\{0\}$, and $a<0$. Then the discrete spectrum of the Hilbert tensor product $N(-1-x+iy,x+iy)\otimes N(a,0)$ is 
$$\displaystyle\bigoplus_{2n<2x-a} N(-1-2x+a+2n,0).$$
\item Let $-1<a_1,\ a_2<0$, and $a<0$. Then the discrete spectrum of the Hilbert tensor product $N(a_1,a_2)\otimes N(a,0)$ is 
 \begin{align*}
\left(\displaystyle\bigoplus_{2n<a_2-a_1-a-1} N(a+a_1-a_2+2n,0)\right)& \oplus N(a+a_1,a_2),\\
& \mbox{if}\ -1<a+a_1+a_2<0,\\
\left(\displaystyle\bigoplus_{2n<a_2-a_1-a-1} N(a+a_1-a_2+2n,0)\right)& \oplus N(a_1,a_2-a),\\
& \mbox{if}\ -2<a_1+a_2-a<-1,\\
\left(\displaystyle\bigoplus_{2n<a_2-a_1-a-1} N(a+a_1-a_2+2n,0)\right),&\ \mbox{otherwise}.
          \end{align*}
\end{enumerate}
\end{theo}
\begin{proof}
 This is a consequence of propositions \ref{prop:discHW}, \ref{prop:discLW}, \ref{prop:discPS}, and \ref{prop:discCS}.
\end{proof}

\begin{rema}
 In \cite{Re76}, Repka gives the decomposition of tensor products of unitary representations of $SU(1,1)$. Theorem \ref{thm:disc-spec} recovers in particular (some of) these results. Note also that the particular case $N(0,a)\otimes N(a,0)$ was obtained by in \cite{OZ97}.
\end{rema}

\subsection{Application to smooth vectors}

Let $a_1,\ a_2,\ a$ be real numbers such that $-1<a_1\leq 0$, $-1<a,a_2<0$ and $-1<a+a_1+a_2<0$. From proposition \ref{prop:discCS}, we know that the (completed) tensor product $V=N(a_1,a_2)\otimes N(a,0)$ contains a (Hilbert) submodule $W$ isomorphic to $N(a_1+a,a_2)$. We want to determine the possible relation between the smooth vectors in $W$ and the smooth vectors in $V$. We denote them respectively by $W_{\infty}$ and $V_{\infty}$.

\begin{prop}\label{prop:Cinfty}
 With notations as above, we have $W\cap V_{\infty}=\{0\}$. In particular, $W_{\infty}\cap V_{\infty}=\{0\}$.
\end{prop}
\begin{proof}
According to Nelson's theorem \ref{thm:nelson}, the set of smooth vectors of a unitarisable module is the common domain of definition of the various operators $\rho(u)$ for $u\in \cal U (\germ g )$. Denote by $\{w(k),\ k\in\ \Z\}$ the standard basis of $W=N(a_1+a,a_2)$. Recall that the action of the triple $(H,E,F)$ on $W$ is given by 
$$\left\{\begin{array}{lcl}
          H\cdot w(k) & = & (a+a_1-a_2+2k)w(k)\\
	  E\cdot w(k) & = & (a_2-k)w(k+1)\\
	  F\cdot w(k) & = & (a+a_1+k)w(k-1)
         \end{array}\right.$$
On the other hand, formulae \eqref{norm:cpl} gives $\|w(k)\|^2 \sim |k|^{1+a+a_1+a_2}\|w(0)\|^2$, and so 
$$W=\left\{\displaystyle\sum_{k\in\ \Z}\ \alpha_kw(k)\ : \ \displaystyle\sum_{k\in\ \Z}\ |\alpha_k|^2|k|^{1+a+a_1+a_2}<\infty\right\}.$$
Hence Nelson's theorem implies that 
$$W_{\infty}=\left\{\displaystyle\sum_{k\in\Z}\ \alpha_kx(k)\ : \ \forall\ N \in\ \Z_{\geq 0},\ \displaystyle\sum_{k\in\Z}\ |\alpha_k|^2|k|^{1+a+a_1+a_2}k^{2N} <\infty\right\}.$$
Recall then that 
$$V=\left\{\displaystyle\sum_{k,l}\ u_{k,l}z(k,l)\ : \ \displaystyle\sum_{k,l}\  |u_{k,l}|^2|k|^{1+a_1+a_2}|l|^{1+a}< \infty\right\}.$$
Using formulae \eqref{action1}, giving the action of $E$ and $F$ on $V$, we check that 
$$V_{\infty}=\left\{\displaystyle\sum_{k,l}\ u_{k,l}z(k,l)\ : \ \forall\ N \in\ \Z_{\geq 0},\ \displaystyle\sum_{k,l}\  |u_{k,l}|^2|k|^{1+a_1+a_2}|l|^{1+a}(k^2+l^2)^N< \infty\right\}.$$

Now, we know that $$w(0)=\displaystyle\sum_{n\geq 0}\ \frac{(-a)_n(-a_2)_n}{(1+a_1)_n}\frac{z(n,n)}{n!}.$$ Then the standard basis is given by the following formulae:
$$\left\{\begin{array}{lcll}
 w(k) & = & \displaystyle\prod_{j=1}^k\ \frac{a+a_1+j}{a_1+j}\times \displaystyle\sum_{n\geq 0}\ \frac{(-a)_n(k-a_2)_n}{(k+1+a_1)_n}\frac{z(k+n,n)}{n!}, & k>0\\
 w(-k) & = & \displaystyle\prod_{j=1}^k\ \frac{a_1+1-j}{a+a_1+1-j}\times \displaystyle\sum_{n\geq 0}\ \frac{(-a)_n(-k-a_2)_n}{(-k+1+a_1)_n}\frac{z(-k+n,n)}{n!}, & k>0.
\end{array}\right.$$
It is now easy to check that $w(k)\not\in\ V_{\infty}$. Remark then that the weight vectors that occur in the decomposition of $w(k)$ and $w(l)$ are all distinct if $k\not= l$. Therefore, we conclude that $W\cap V_{\infty}=\{0\}$, as asserted.
\end{proof}

\begin{rema}
 In \cite{SV10}, Speh and Venkataramana proved an analogous result about the $K$-finite vectors for the restriction of some complementary series of $SO(n,1)$ to $SO(n-1,1)$.
\end{rema}


\bibliographystyle{abbrv}


\end{document}